\documentclass[amssymb,twoside,11pt]{article}
\usepackage{latexsym,amsmath,graphicx}
\usepackage[dvipsnames]{xcolor}
\usepackage{placeins}
\usepackage{amsfonts}
\usepackage{enumitem}
\newtheorem{theorem}{Theorem}[section]
\newtheorem{Lemma}[theorem]{{\bf Lemma}}

\newtheorem{rem}[theorem]{{\bf Remark}}
\newtheorem{ex}[theorem]{{\bf Example}}
\newtheorem{definition}{Definition}[section]
\numberwithin{equation}{section}
\newenvironment{proof}{\indent{\em Proof:}}{\quad \hfill
$\Box$\vspace*{2ex}}

\font\Bbb=msbm10 at 12pt

\newcommand{\R}{\mbox{\Bbb R}}

\begin{document}
\setcounter{page}{1}
\begin{center}
\vspace{0.4cm} {\large{\bf  On Coupled  System of Nonlinear  $\Psi$-Hilfer Hybrid\\ Fractional  Differential Equations}} \\
\vspace{0.5cm}
Ashwini D. Mali  $^{1}$\\
maliashwini144@gmail.com\\

\vspace{0.35cm}
Kishor D. Kucche $^{2}$ \\
kdkucche@gmail.com \\

\vspace{0.35cm}
J. Vanterler da C. Sousa$^{3}$\\
vanterlermatematico@hotmail.com, jose.vanterler@edu.ufabc.br\\

\vspace{0.35cm}
$^{1,2}$ Department of Mathematics, Shivaji University, Kolhapur-416 004, Maharashtra, India.\\
$^{3}$ Centro de Matemática, Computação e Cognição, Universidade Federal do ABC,\\ Avenida dos Estados, 5001, Bairro Bangu, 09.210-580, Santo André, SP - Brazil
\vspace{0.5cm}

\end{center}

\def\baselinestretch{1.0}\small\normalsize

\begin{abstract}

This paper is dedicated to investigating the existence of solutions to the initial value problem (IVP) for a coupled system of $\Psi$-Hilfer hybrid fractional differential equations (FDEs) and boundary value problem (BVP) for a coupled system of $\Psi$-Hilfer hybrid FDEs. Analysis of the current paper depends on the two fixed point theorems involving three operators characterized on Banach algebra. In the view of an application, we provided concrete examples to exhibit the effectiveness of our achieved results. 

\end{abstract}
\noindent\textbf{Key words:} Coupled fractional differential equations;  $\Psi$-Hilfer fractional derivative; Existence of solution, Initial value problem; Boundary value problem.\\
\noindent
\textbf{2020 Mathematics Subject Classification:} 26A33, 34A38,  34A12, 34A08.
\def\baselinestretch{1.5}
\allowdisplaybreaks
\section{Introduction}
In many situations, a nonlinear differential equation may not be analyzed in its original form for the existence of a solution or to examine distinctive qualitative properties of the solution. In such cases, the perturbation of the original differential equation makes it simple to analyze various properties of the solution. Motivated by this reality, Dhage and Lakshmikantham \cite{Dhage1} initiated the study of quadratic perturbation of the second type which is well known as hybrid nonlinear differential equations. The fractional counterpart of hybrid differential equations with Riemann--Liouville (RL) derivative have been analyzed by Zhao et al. \cite{Zhao}, developed fractional differential inequalities, obtained existence of extremal solutions and proved comparison theorems. 

With the growth and greater demand for the theory of fractional differential equations (FDEs), the search for discussing properties of solutions of hybrid differential problems, has gained prominence and greater investigation both in the theoretical sense and involving applications \cite{Herzallah,HybAhmad2,Sitho,Caballero,Mahmudov}. For further reading that involves the existence of solutions of hybrid FDEs, we recommend  \cite{Mali4,  Lachouri, Matarhyb}.  Few fundamental works on boundary value problem of hybrid FDEs can be found in \cite{Sun1,HybAhmad3, Mohammadi,  Ahmadhyb1, Ji}. Numerous specialists have analyzed coupled hybrid FDEs  from a different point of view and researched the existence and uniqueness of the solution \cite{Bashirihyb,Alihyb,Baleanuhyb, Dhagehyb1, Ahmadhyb2, Buvaneswari}.

On the other hand,  the FDEs involving the most generalized fractional differential operator called $\Psi$-Hilfer fractional derivative \cite{Vanterler1} has attracted considerable attention from researchers. The basic analysis of various class of nonlinear $\Psi$-Hilfer FDEs relating to the existence and uniqueness of the solution, Ulam-Hyers stability, comparison theorems, extremal solution and comparison result concerning lower and upper solutions can be found in \cite{Mali1, Mali2, Mali3, Mali5, Kharade, MS Abdo2, Ahmadhyb, Jose1, Luo}.

The importance of hybrid FDEs in the field of nonlinear analysis leads to a reestablished flow of research activity in the area of FDEs. Motivated by this fact and the work referenced above, in the current paper, we consider the following two kinds of coupled hybrid FDEs involving the most generalized fractional differential operator called $\Psi$-Hilfer fractional derivative.
\begin{itemize}
\item Initial value problem (IVP) for coupled system of $\Psi$-Hilfer hybrid FDEs:\\
\begin{align}~\label{eqq611}
\begin{cases}
& ^H \mathcal{D}^{\mu,\,\nu\,;\, \Psi}_{0^+}\left[ \dfrac{y(t)-w(t, y(t))}{ u(t, y(t))}\right] 
  = v\left(t, x(t), k\, \mathcal{I} ^{\mu\,;\, \Psi}_{0^+}x(t)\right),~a.e. ~t \in  (0,\,T],  \\
    & \lim\limits_{t\to 0+}\left( \Psi \left( t\right) -\Psi \left( 0\right) \right)^{1-\xi }y(t)=y_{0} \in\R,\\
 \end{cases}
  \end{align}
  and
  \begin{align}\label{eqq612}
 \begin{cases}  
   & ^H \mathcal{D}^{\mu,\,\nu\,;\, \Psi}_{0^+}\left[ \dfrac{x(t)-w(t, x(t))}{ u(t, x(t))}\right] 
      = v\left(t, y(t), k\, \mathcal{I} ^{\mu\,;\, \Psi}_{0^+}  y(t)\right),~a.e. ~t \in  (0,\,T],  ~\\  
  & \lim\limits_{t\to 0+}\left( \Psi \left( t\right) -\Psi \left( 0\right) \right)^{1-\xi }x(t)=y_{0} \in\R,
  \end{cases}
 \end{align}
where $ 0<\mu<1, 0\leq\nu\leq 1,\,\xi=\mu+\nu(1-\mu)(0<\xi\leq 1)$, $^H \mathcal{D}^{\mu,\nu;\, \Psi}_{0^+}(\cdot)$ is the $\Psi$-Hilfer fractional derivative of order $\mu$ and type $\nu$, $u\in C(J \times \R  \,, \R\setminus\{0\}) $, $J=[0,T]$, $w\in C(J \times \R  \,, \R)$ and $v\in C(J \times \R\times \R  \,, \R) $.
\item Boundary value problem (BVP) for coupled system of $\Psi$-Hilfer hybrid FDEs:\\
\begin{align}\label{eq663}
 \begin{cases}
& ^H \mathcal{D}^{\mu,\,\nu\,;\, \Psi}_{0^+}\left[ \dfrac{y(t)-w_1(t, y(t), x(t))}{ u_1(t, y(t), x(t))}\right] 
= v_1\left(t, y(t), x(t)\right),~a.e. ~t \in  (0,\,T],  ~\\
& a\,\lim\limits_{t\to 0+}\left( \Psi \left( t\right) -\Psi \left( 0\right) \right)^{1-\xi }y(t)+b\,\lim\limits_{t\to T} \left( \Psi \left( t\right) -\Psi \left( 0\right) \right)^{1-\xi }y(t)=y_{0} \in\R,\\
\end{cases}
\end{align}
and
\begin{align}\label{eq664}
\begin{cases}  
& ^H \mathcal{D}^{\mu,\,\nu\,;\, \Psi}_{0^+}\left[ \dfrac{x(t)-w_2(t, y(t), x(t))}{ u_2(t, y(t), x(t))}\right] 
= v_2\left(t, y(t), x(t)\right),~a.e. ~t \in  (0,\,T],  ~\\
& a\,\lim\limits_{t\to 0+}\left( \Psi \left( t\right) -\Psi \left( 0\right) \right)^{1-\xi }x(t)+b\,\lim\limits_{t\to T} \left( \Psi \left( t\right) -\Psi \left( 0\right) \right)^{1-\xi }x(t)=y_{0} \in\R,
\end{cases}
\end{align} 
where $0<\mu<1, 0\leq\nu\leq 1,\, \xi=\mu+\nu(1-\mu)(0<\xi\leq 1)$,\,$^H \mathcal{D}^{\mu,\nu;\, \Psi}_{0^+}(\cdot)$ is the $\Psi$-Hilfer fractional derivative of order $\mu$ and type $\nu$, $J=[0,T]$,  $a\neq0$ and $b\in \R $ are the constants, $u_i\in C(J \times \R  \times \R \,, \R\setminus\{0\}) (i=1,2)$,    $w_i\in C(J \times \R  \times \R \,, \R)(i=1,2)$ and $v_i\in C(J \times \R\times \R  \,, \R)(i=1,2) $.
\end{itemize}

We investigate the existence of solutions to IVP for a coupled system of nonlinear $\Psi$-Hilfer hybrid {\rm {\rm FDEs}} \eqref{eqq611}-\eqref{eqq612}. Next, we derive the equivalent fractional integral equation to the  {\rm BVPs} for coupled system of $\Psi$-Hilfer hybrid {\rm FDEs} {\rm\eqref{eq663}-\eqref{eq664}} and establish an existence result for it. The existence results are established through the fixed point theorems involving three operators characterized on Banach algebra. Finally, we provide concrete examples in support of the results we derived.

It is seen that our principle results incorporate the investigation of the following research work in the literature: 
\begin{itemize}
\item  For $\nu=0, \Psi(t)=t, y_0=0$ and $k=0$,  the coupled system \eqref{eqq611}-\eqref{eqq612} includes the study of Bashiri et al. \cite{Bashirihyb} for the hybrid FDEs involving RL fractional derivative of the  form
\begin{align*}
\begin{cases}
& ^{RL} \mathcal{D}^{\mu}_{0^+}\left[ \dfrac{y(t)-w(t, y(t))}{ u(t, y(t))}\right] 
  = v\left(t, x(t)\right),~a.e. ~t \in  (0,\,T],  \\
   & ^{RL} \mathcal{D}^{\mu}_{0^+}\left[ \dfrac{x(t)-w(t, x(t))}{ u(t, x(t))}\right] 
      = v\left(t, y(t)\right),~a.e. ~t \in  (0,\,T],  ~\\  
  &y(0)=0,\quad x(0)=0.
  \end{cases}
 \end{align*}
\item For $\nu=1, \Psi(t)=t, u_1=0, u_2=0,\,  a=1, b=0$ and $y_0=0$,  the coupled system \eqref{eq663}-\eqref{eq664}  includes the study of Shah and Khan \cite{Shahhyb} for the hybrid FDEs involving Caputo fractional derivative ($\sigma=\rho=\mu$) of the  form
\begin{align*}
 \begin{cases}
& ^C \mathcal{D}^{\mu}_{0^+}\left[ y(t)-w_1(t, y(t), x(t))\right] 
= v_1\left(t, y(t), x(t)\right),~a.e. ~t \in  (0,\,T],  ~\\
& ^C \mathcal{D}^{\mu}_{0^+}\left[ x(t)-w_2(t, y(t), x(t))\right] 
= v_2\left(t, y(t), x(t)\right),~a.e. ~t \in  (0,\,T],  ~\\
&y(t)|_{t=0}=0, \quad x(t)|_{t=0}=0.
\end{cases}
\end{align*}

\end{itemize}
   
The structure of this paper is as follows. In section 2, we review a few essentials of $\Psi$-Hilfer fractional derivative and fixed point theorems for coupled frameworks. In section 3, we demonstrate an existence result for the coupled system of nonlinear $\Psi$-Hilfer hybrid {\rm {\rm FDEs}} \eqref{eqq611}-\eqref{eqq612}. Section 4 deals with the obtaining equivalent fractional integral equation to the  {\rm BVPs} for coupled system of $\Psi$-Hilfer hybrid {\rm FDEs} {\rm \eqref{eq663}-\eqref{eq664}} and to establish an existence result for it. In section 5,  two examples are provided to support the acquired outcomes.

\section{Preliminaries} \label{preliminaries} 
Let $[a,b]$ $(0<a<b<\infty)$ be a finite interval and $\Psi\in C^{1}([a,b],\mathbb{R})$ be an increasing function such that $\Psi'(t)\neq 0$, for all $~ t\in [a,b]$.  We consider the  weighted space \cite{Vanterler1} 
\begin{equation*}
C_{1-\xi ;\, \Psi }\left[ a,b\right] =\left\{ h\big| h:\left( a,b\right]
\rightarrow \mathbb{R}, h(a+) ~\text{exists}~ \text{and}~\left( \Psi \left( t\right) -\Psi \left(
a\right) \right) ^{1-\xi }h\left( t\right) \in C\left[ a,b\right]
\right\} ,\text{ }0< \xi \leq 1,
\end{equation*}
endowed with the norm
\begin{equation}\label{space1}
\left\Vert h\right\Vert _{C_{1-\xi ;\Psi }\left[ a,b\right] }=\underset{t\in \left[ a,b\right] 
}{\max }\left\vert \left( \Psi \left( t\right) -\Psi \left( a\right) \right)
^{1-\xi }h\left( t\right) \right\vert.
\end{equation}

\begin{definition} [\cite{Kilbas}]
Let  $h$ be an integrable function defined on $[a,b]$. Then the $\Psi$-Riemann-Liouville fractional integral of order $\mu>0 ~(\mu \in \R)$ of the function $h$ is given by 
\begin{equation}\label{P1}
I_{a^+}^{\mu \, ;\Psi }h\left( t\right) =\frac{1}{\Gamma \left( \mu
\right) }\int_{a}^{t}\Psi ^{\prime }\left( s\right) \left( \Psi \left(
t\right) -\Psi \left( s\right) \right) ^{\mu -1}h\left( s\right) ds.
\end{equation}
\end{definition}

\begin{definition} [\cite{Vanterler1}] The $\Psi$-Hilfer fractional derivative of a function $h$ of order $0<\mu<1$ and type $0\leq \nu \leq 1$, is defined by
$$^H \mathcal{D}^{\mu, \, \nu; \, \Psi}_{a^+}h(t)= I_{a^+}^{\nu ({1-\mu});\, \Psi} \left(\frac{1}{{\Psi}^{'}(t)}\frac{d}{dt}\right)I_{a^+}^{(1-\nu)(1-\mu);\, \Psi} h(t).$$
\end{definition}

\begin{Lemma}[\cite{Vanterler1, Kilbas}]\label{lema2} 
Let $\chi, \delta>0$ and $\rho>n$.  Then
\begin{enumerate}[topsep=0pt,itemsep=-1ex,partopsep=1ex,parsep=1ex]

\item [(i)] $\mathcal{I}_{a^+}^{\mu \, ;\,\Psi }\mathcal{I}_{a^+}^{\chi \, ;\,\Psi }h(t)=\mathcal{I}_{a^+}^{\mu+\chi \, ;\,\Psi }h(t)$.

\item [(ii)]$
\mathcal{I}_{a^+}^{\mu \, ;\,\Psi }\left( \Psi \left( t\right) -\Psi \left( a\right) \right) ^{\delta -1}=\dfrac{\Gamma \left( \delta \right) }{\Gamma \left(\mu +\delta \right) }\left( \Psi \left(t\right) -\Psi \left( a\right) \right) ^{\mu +\delta -1}.
$
\item [(iii)]$^{H}\mathcal{D}_{a^+}^{\mu ,\,\nu \, ;\,\Psi }\left( \Psi \left(t\right) -\Psi \left( a\right)
\right) ^{\xi -1}=0.
$
\end{enumerate}
\end{Lemma}

\begin{Lemma}[\cite{Vanterler1}]\label{teo1} 
If $h\in C^{n}[a,b]$, $n-1<\mu<n$ and $0\leq \nu \leq 1$, then
\begin{enumerate}[topsep=0pt,itemsep=-1ex,partopsep=1ex,parsep=1ex]

\item [(i)]$
I_{a^+}^{\mu \, ;\Psi }\text{ }^{H}\mathcal{D}_{a^+}^{\mu ,\nu \, ;\Psi }h\left( t\right) =h\left(t\right) -\overset{n}{\underset{k=1}{\sum }}\dfrac{\left( \Psi \left( t\right) -\Psi \left( a\right) \right) ^{\xi -k}}{\Gamma \left( \xi -k+1\right) }h_{\Psi }^{\left[ n-k\right] }I_{a^+}^{\left( 1-\nu \right) \left( n-\mu \right) \, ;\Psi }h\left( a\right)$,\\
where $h_{\Psi }^{\left[ n-k\right] }h(t)=\left( \dfrac{1}{\Psi'(t)}\dfrac{d}{dt}\right)^{n-k}h(t)$.

\item [(ii)]$
^{H}\mathcal{D}_{a^+}^{\mu ,\nu \, ;\Psi }I_{a^+}^{\mu \, ;\Psi }h\left( t\right)
=h\left( t\right).
$
\end{enumerate}
\end{Lemma}

\begin{Lemma}[\cite{Mali4}]\label{malihyb}
Let $0<\mu<1, ~0\leq\nu\leq 1,~ \xi=\mu+\nu(1-\mu)$, $f\in C(J \times \R  \,, \R\setminus\{0\}) $ is bounded, $J=[0,T]$ and   
$g\in \mathfrak{C}(J \times \R  \,, \R)= \{ h ~|~ \text{the map}~\omega \to h(\tau,\omega)\; \text{is  continuous } \;\text{for each} ~\tau$ and the map $\tau \to h(\tau,\omega) \;\text{is measurable} \;\text{for each} ~\omega \}$. A  function $y\in C_{1-\xi ;\, \Psi }(J,\,\R)$ is the solution of hybrid {\rm FDEs}
\begin{align}
& ^H \mathcal{D}^{\mu,\,\nu\,;\, \Psi}_{0^+}\left[ \frac{y(t)}{ f(t, y(t))}\right] 
 = g(t, y(t)),~a.e. ~t \in  (0,\,T],  ~\label{eqq1}\\
& \left( \Psi \left( t\right) -\Psi \left( 0\right) \right)^{1-\xi }y(t)|_{t=0}=y_{0} \in\R ,\label{eqq2}
\end{align}
if and only if it is solution of the following hybrid fractional integral equation {\rm (IE)}
 \begin{align}\label{a1}
 y(t)&=f(t,y(t))\left\lbrace  \frac{y_{0}}{f(0,y(0+))}\left( \Psi \left( t\right) -\Psi \left( 0\right) \right)^{\xi-1 }+\mathcal{I}_{0^+}^{\mu\,;\, \Psi}g(t,y(t))\right\rbrace ,~t\in(0,T].
 \end{align}
\end{Lemma}

\begin{definition} [\cite{Chang}] An element $(x, y )\in X\times X $ is called a coupled fixed point of a mapping $T:X\times X\rightarrow X$ if $T(x, y)=x$ and $T(y, x)=y$.
\end{definition}

\begin{Lemma}[\cite{Bashirihyb}] \label{hyb2}
Let  $S$ be a non-empty, closed, convex and bounded    subset of  the   Banach algebra $X$  and $\tilde{S}=S\times S$. Suppose that  $E, G:X\rightarrow X $ and $F:S\rightarrow X $ are three operators such that
\begin{itemize}[topsep=0pt,itemsep=-1ex,partopsep=1ex,parsep=1ex]

\item [(a)] $E$ and $G$ are Lipschitzian with a Lipschitz constants $\sigma$ and $\delta$ respectively;

\item [(b)] $F$ is completely continuous;

\item [(c)] $y=Ey\,Fx+Gy \implies y\in S~ \text{for all} ~x\in S$ and

\item [(d)] $4\,\sigma M+\delta <1 $ where $M=\sup\left\lbrace \left\|Bx \right\|: x\in  S \right\rbrace $.
\end{itemize}

Then, the operator equation $T(y, x)=Ey\,Fx+Gy$ has a at least one coupled fixed point  in $\tilde{S}$.
\end{Lemma}

\begin{Lemma}[\cite{Dhagehyb3}] \label{hyb3}
Let  $S^*$ be a non-empty, closed, convex and bounded    subset of  the   Banach space $E$ and let  $A, C:E\rightarrow E$ and $B:S^*\rightarrow E $ are three operators such that
\begin{itemize}[topsep=0pt,itemsep=-1ex,partopsep=1ex,parsep=1ex]

\item [(a)] $A$ and $C$ are Lipschitzian with a Lipschitz constants $K$ and $L$ respectively;

\item [(b)] $B$ is completely continuous;

\item [(c)] $y=Ay\,Bx+Cy \implies y\in S^*~ \text{for all} ~x\in S^*$ and

\item [(d)] $K\, M^*+ \,L <1 $ where $M^*=\sup\left\lbrace \left\|By \right\|: y\in  S^* \right\rbrace $.
\end{itemize}

Then, the operator equation $Ay\,By + Cy=y$ has a solution  in $S^*$.
\end{Lemma}
\section{{\rm IVP} for Coupled system of Hyrid {\rm FDEs}}

An application of the Lemma \ref{malihyb} gives the equivalent fractional IE to the {\rm FDEs} \eqref{eqq611},  given in the following Lemma.

\begin{Lemma}\label{lem61}
A  function $y\in C_{1-\xi ;\, \Psi }\left( J,\mathbb{R}\right) $ is the solution of  the Cauchy problem for hybrid {\rm FDEs} 
\begin{align*}
\begin{cases}
& ^H \mathcal{D}^{\mu,\,\nu\,;\, \Psi}_{0^+}\left[ \dfrac{y(t)-w(t, y(t))}{ u(t, y(t))}\right] 
  = v\left(t, x(t), k\, \mathcal{I} ^{\mu\,;\, \Psi}_{0^+}x(t)\right),~a.e. ~t \in  (0,\,T],  \\
    & \lim\limits_{t\to 0+}\left( \Psi \left( t\right) -\Psi \left( 0\right) \right)^{1-\xi }y(t)=y_{0} \in\R,\\
 \end{cases}
  \end{align*}
if and only if it is solution of the following hybrid fractional IE
\begin{align*}
y(t)&= u(t,y(t))\left\lbrace  \frac{y_{0}}{u(0,y(0+))}\left( \Psi \left( t\right) -\Psi \left( 0\right) \right)^{\xi-1 }+\mathcal{I}_{0^+}^{\mu\,;\, \Psi}v\left(t, x(t), k\, \mathcal{I} ^{\mu\,;\, \Psi}_{0^+}x(t)\right)\right\rbrace\nonumber\\
&\qquad+ w(t,y(t)),~t\in(0,T].
\end{align*}
\end{Lemma}

We list the following  assumptions to prove the existence of solution  to the  coupled system of hybrid {\rm FDEs} \eqref{eqq611}-\eqref{eqq612}.
\begin{enumerate}[topsep=0pt,itemsep=-1ex,partopsep=1ex,parsep=1ex]
\item [{\bf (H1)}] The functions $u\in C\left(J\times\R, \R\setminus\{0\} \right) $ and $w\in C\left(J\times\R, \R \right) $ are bounded  and there exists constants $\sigma, \delta>0$ such that for all $p, q\in \R$ and $t\in J=[0, T]$, we have 
$$\left| u(t,p) - u(t,q)\right| \leq \sigma \left| p-q\right|$$
and 
$$\left| w(t,p) - w(t,q)\right| \leq \delta \left| p-q\right|.$$
  
\item [{\bf (H2)}] The function $ v \in C (J \times \R \times \R  \,, \R) $  and there exists a function $ g\in C_{1-\xi;\Psi}\left( J,\mathbb{R}\right)  $ such that
   $$
    \left| v(t, p, q)\right| \leq \left( \Psi \left( t\right) -\Psi \left( 0\right) \right)^{1-\xi } g(t),\,~a.e.\,\, t\in J~ \text{and}~  \,p, q\in \R.
   $$
   \end{enumerate}
 
 \title{Existence Theorem} 
\begin{theorem}\label{tha63.2} 
Assume that the  hypotheses {\normalfont{\bf (H1)}-{\bf (H2)}} hold. Then, the coupled system of nonlinear $\Psi$-Hilfer hybrid  {\rm FDEs} \eqref{eqq611}-\eqref{eqq612} has a solution $(y, x)\in C_{1-\xi ;\, \Psi }\left( J,\mathbb{R}\right)\times C_{1-\xi ;\, \Psi }\left( J,\mathbb{R}\right) $ provided
\begin{equation}\label{616}
4\, \sigma \left\lbrace \left| \frac{y_0\,}{u(0,y(0+))}\right| + \frac{\left( \Psi \left( T\right) -\Psi \left( 0\right) \right)^{\mu+1-\xi }  }{\Gamma(\mu+1)}\left\| g\right\|_{C_{1-\xi ;\, \Psi }\left( J,\mathbb{R}\right) } \right\rbrace +\delta <1.
\end{equation}
 \end{theorem}
\begin{proof} Let $X:=\left( C_{1-\xi ;\, \Psi }\left( J,\mathbb{R}\right) , \,\left\Vert \cdot\right\Vert _{C_{1-\xi ;\,\Psi }\left( J,\mathbb{R}\right) }\right) $.  Then $X$ is a Banach algebra with the product of vectors defined by $(xy)(t)=x(t)y(t),\,t\in  (0,T]$. Define,
$$
S=\{x\in X: \left\Vert x\right\Vert _{C_{1-\xi ;\,\Psi }\left( J,\mathbb{R}\right)  }\leq R\},
$$
where 
$$
R=K_1\left\lbrace \left| \frac{y_0\,}{u(0,x(0+))}\right| + \frac{\left( \Psi \left( T\right) -\Psi \left( 0\right) \right)^{\mu+1-\xi }  }{\Gamma(\mu+1)}\left\| g\right\|_{C_{1-\xi ;\, \Psi }\left( J,\mathbb{R}\right) } \right\rbrace + K_2\, \left( \Psi \left( T\right) -\Psi \left( 0\right) \right)^{1-\xi } 
$$
and $K_1>0$ and  $K_2>0$ are the constants such that   $\left| u(t, \cdot)\right| <K_1$  and $\left| w(t, \cdot)\right| <K_2$ for all $t\in J$.

Clearly, $S$ is non-empty, closed, convex and bounded subset of $X$.
If  $(y, x)\in S\times S=\tilde{S}$ is a solution of the coupled system of nonlinear $\Psi$-Hilfer hybrid {\rm FDEs} \eqref{eqq611}-\eqref{eqq612}, then  $(y, x)\in S\times S=\tilde{S}$ is a solution of the coupled system of fractional IEs
\begin{align}   \label{617}
\begin{cases}
y(t)&= u(t, y(t))\left\lbrace  \dfrac{y_{0}}{u(0, y(0+))}\left( \Psi \left( t\right) -\Psi \left( 0\right) \right)^{\xi-1 }+\mathcal{I}_{0^+}^{\mu\,;\, \Psi}v\left(t, x(t), k\, \mathcal{I} ^{\mu\,;\, \Psi}_{0^+}x(t)\right)\right\rbrace\\
 &\quad+ w(t, y(t)),\\
x(t)&= u(t, x(t))\left\lbrace  \dfrac{y_{0}}{u(0,x(0+))}\left( \Psi \left( t\right) -\Psi \left( 0\right) \right)^{\xi-1 }+ \mathcal{I}_{0^+}^{\mu\,;\, \Psi}
v\left(t, y(t), k\, \mathcal{I} ^{\mu\,;\, \Psi}_{0^+}  y(t)\right)\right\rbrace\\
&\quad+ w(t, x(t)),~t\in(0,T].
\end{cases}  
\end{align}  
   
Define three operators  $E, G:X\rightarrow X $ and $F:S\rightarrow X $ by
\begin{align*}
& Ey(t)= u(t,y(t)), \,t\in J;\\
&Fy(t)= \frac{y_0\,}{u(0,y(0+))} \left( \Psi \left( t\right) -\Psi \left( 0\right) \right)^{\xi-1 } + \mathcal{I}_{0^+}^{\mu\,;\, \Psi}v\left(t, y(t), k\, \mathcal{I} ^{\mu\,;\, \Psi}_{0^+}  y(t)\right), \, t\in (0,T];\\
& Gy(t)= w(t,y(t)), \,t\in J.
\end{align*}

Then, the coupled hybrid IEs in Eq.\eqref{617}  transformed into the coupled system of operator equations as 
\begin{align}\label{618}
\begin{cases}
y&= Ey\,Fx +Gy,\, y\in X,\\
x&= Ex\,Fy +Gx,\, x\in X.   
\end{cases}
\end{align}
Consider the mapping $T: \tilde{S}\to X, \,\tilde{S}=S\times S $ defined by    
$$
T(y, x)= Ey\,Fx +Gy, ~(y, x)\in \tilde{S}.
$$ 
Then the coupled system of operator  equations  \eqref{618} can be written as
$$
y=T(y, x)~\text{and}~x=T(x, y),~(y, x), (x, y)\in \tilde{S}.
$$ To prove that the mapping $T$ has coupled fixed point, we show that  
the operators $E$, $F$ and $G$ satisfies all the conditions of Lemma \ref{hyb2}. The proof is given in the several steps:

\textbf{Step 1:} $E, G:X\rightarrow X$  are Lipschitz operators.

Using the hypothesis {\bf (H1)}, we obtain
\begin{align*}
\left| \left( \Psi \left( t\right) -\Psi \left( 0\right) \right)^{1-\xi} \left(  Ex(t)-Ey(t) \right)   \right|
 &=\left| \left( \Psi \left( t\right) -\Psi \left( 0\right) \right)^{1-\xi }\left(  u(t,x(t))-u(t,y(t)) \right) \right| \\
 &\leq \sigma \left|\left( \Psi \left( t\right) -\Psi \left( 0\right) \right)^{1-\xi } \left( x(t)-y(t)\right) \right|\\
 &\leq \sigma \left\Vert x-y\right\Vert _{C_{1-\xi ;\,\Psi }\left( J,\mathbb{R}\right)  }.
\end{align*}
 This gives,
\begin{equation*}
\left\Vert Ex - Ey\right\Vert _{C_{1-\xi ;\,\Psi }\left( J,\mathbb{R}\right)  }\leq \sigma \left\Vert x-y\right\Vert _{C_{1-\xi ;\,\Psi }\left( J,\mathbb{R}\right)  }. 
\end{equation*}

Therefore, $E$ is Lipschitz operator with  Lipschitz constant $\sigma$. On the similar line one can verify    that  $G$ is Lipschitz operator. Let $\delta$ is  Lipschitz constant corresponding to operator $G$.

\textbf{Step 2:} $F:S\rightarrow X$  is completely continuous.

(i) $F:S\rightarrow X$ is continuous.

Let  $\{y_n\} $ be any sequence in $S$ such that $y_n \rightarrow y$ as $n\rightarrow \infty$  in $S$. We prove that  $Fy_n \rightarrow Fy $ as $n\rightarrow \infty$  in $S$.  Consider,
\begin{align*}
\left\Vert Fy_n-Fy\right\Vert _{C_{1-\xi ;\,\Psi }\left( J,\mathbb{R}\right)  }\nonumber
&=\underset{t\in J 
}{\max }\left\vert \left( \Psi \left( t\right) -\Psi \left( 0\right) \right)
^{1-\xi }\left(   Fy_n(t)-Fy(t)\right)   \right\vert\nonumber\\
&\leq \underset{t\in J }{\max }\frac{\left( \Psi \left( t\right) -\Psi \left( 0\right) \right)
^{1-\xi }}{\Gamma \left( \mu\right) } \int_{0}^{t}\Psi'(s)(\Psi(t)-\Psi(s))^{\mu-1} \times\\
&\quad\left\vert  v\left( s, y_n(s), k\, \mathcal{I} ^{\mu\,;\, \Psi}_{0^+}  y_n(s)\right) -v\left( s, y(s), k\, \mathcal{I} ^{\mu\,;\, \Psi}_{0^+}  y(s)\right)  \right\vert\,ds.
\end{align*}

By continuity of $v$ and Lebesgue dominated convergence theorem, from the above inequality,  we obtain
$$
\left\Vert Fy_n-Fy\right\Vert _{C_{1-\xi ;\,\Psi }\left( J,\mathbb{R}\right)  }\rightarrow 0 ~\text{ as}~ n\rightarrow\infty.
$$
This proves $F:S\rightarrow X$ is continuous.

(ii) $F(S)=\left\lbrace Fy: y\in S\right\rbrace $ is uniformly bounded.

Using hypothesis {\bf (H2)}, for any $y\in S$ and $t\in J$, we have
\begin{align*}
&\left\vert \left( \Psi \left( t\right) -\Psi \left( 0\right) \right)
^{1-\xi }Fy(t) \right\vert\\
&\leq \left\vert \frac{y_0}{u(0,y(0+))}\right\vert +\frac{\,\left( \Psi \left( t\right) -\Psi \left( 0\right) \right)^{1-\xi }}{\Gamma \left( \mu\right) }\int_{0}^{t}\Psi'(s)(\Psi(t)-\Psi(s))^{\mu-1} \left\vert v\left(s, y(s), k\, \mathcal{I} ^{\mu\,;\, \Psi}_{0^+}  y(s)\right)\right\vert \,ds\\
&\leq \left\vert \frac{y_0}{u(0,y(0+))}\right\vert +\frac{\,\left( \Psi \left( t\right) -\Psi \left( 0\right) \right)^{1-\xi }}{\Gamma \left( \mu\right) }\int_{0}^{t}\Psi'(s)(\Psi(t)-\Psi(s))^{\mu-1} (\Psi(s)-\Psi(0))^{1-\xi} g(s)\,ds\\
&\leq \left\vert \frac{y_0}{u(0,y(0+))}\right\vert +\left\| g\right\|_{C_{1-\xi ;\, \Psi }\left( J,\mathbb{R}\right) }\,\left( \Psi \left( t\right) -\Psi \left( 0\right) \right)^{1-\xi }\frac{\left( \Psi \left( t\right) -\Psi \left( 0\right) \right)^{\mu }}{\Gamma \left( \mu+1     \right) }   \\    
&\leq \left\vert \frac{y_0}{u(0,y(0+))}\right\vert +\frac{\left( \Psi \left( T\right) -\Psi \left( 0\right) \right)^{\mu+1-\xi }}{\Gamma \left( \mu+1     \right) } \left\| g\right\|_{C_{1-\xi ;\, \Psi }\left( J,\mathbb{R}\right) }. 
\end{align*}

Therefore, 
\begin{equation}\label{619}
\left\Vert Fy\right\Vert _{C_{1-\xi ;\,\Psi }\left( J,\mathbb{R}\right)  }\leq \left\vert \frac{y_0}{u(0,y(0+))}\right\vert +\frac{\left( \Psi \left( T\right) -\Psi \left( 0\right) \right)^{\mu+1-\xi }}{\Gamma \left( \mu+1 \right) } \left\| g\right\|_{C_{1-\xi ;\, \Psi }\left( J,\mathbb{R}\right) }.
\end{equation}

(iii) $F(S)$ is equicontinuous. 

Let any $y\in S$ and $t_1, t_2\in J $ with $t_1<t_2$. Then using hypothesis {\bf (H2)}, we have
\begin{align*}
&\left\vert \left( \Psi \left( t_2\right) -\Psi \left( 0\right) \right)
^{1-\xi }Fy(t_2)-\left( \Psi \left( t_1\right) -\Psi \left( 0\right) \right) ^{1-\xi }Fy(t_1) \right\vert\\
&=\left\vert\left\lbrace \frac{y_0}{u(0,y(0+))} +\frac{\left( \Psi \left( t_2\right) -\Psi \left( 0\right) \right)
^{1-\xi } }{\Gamma \left( \mu\right) }\int_{0}^{t_2}\Psi'(s)(\Psi(t_2)-\Psi(s))^{\mu-1} v\left(s, y(s), k\, \mathcal{I} ^{\mu\,;\, \Psi}_{0^+}  y(s)\right)\,ds\right\rbrace\right.\\
&\left.-\left\lbrace \frac{y_0}{u(0,y(0+))} +\frac{\left( \Psi \left( t_1\right) -\Psi \left( 0\right) \right)
^{1-\xi }}{\Gamma \left( \mu\right) }\int_{0}^{t_1}\Psi'(s)(\Psi(t_1)-\Psi(s))^{\mu-1} v\left(s, y(s), k\, \mathcal{I} ^{\mu\,;\, \Psi}_{0^+}  y(s)\right)\,ds\right\rbrace\right\vert\\
&\leq 
\left\vert   \frac{\left( \Psi \left( t_2\right) -\Psi \left( 0\right) \right)
^{1-\xi } }{\Gamma \left( \mu\right) }\int_{0}^{t_2}\Psi'(s)(\Psi(t_2)-\Psi(s))^{\mu-1} \left| v\left(s, y(s), k\, \mathcal{I} ^{\mu\,;\, \Psi}_{0^+}  y(s)\right)\right| \,ds\right.\\
&\left.\qquad-\frac{\left( \Psi \left( t_1\right) -\Psi \left( 0\right) \right)
^{1-\xi } }{\Gamma \left( \mu\right) }\int_{0}^{t_1}\Psi'(s)(\Psi(t_1)-\Psi(s))^{\mu-1} \left| v\left(s, y(s), k\, \mathcal{I} ^{\mu\,;\, \Psi}_{0^+}  y(s)\right)\right|\,ds\right\vert\\
&\leq 
\left\vert   \frac{\left( \Psi \left( t_2\right) -\Psi \left( 0\right) \right)
^{1-\xi } }{\Gamma \left( \mu\right) }\int_{0}^{t_2}\Psi'(s)(\Psi(t_2)-\Psi(s))^{\mu-1}  (\Psi(s)-\Psi(0))^{1-\xi} g(s)\,ds\right.\\
&\left.\qquad-\frac{\left( \Psi \left( t_1\right) -\Psi \left( 0\right) \right)
^{1-\xi } }{\Gamma \left( \mu\right) }\int_{0}^{t_1}\Psi'(s)(\Psi(t_1)-\Psi(s))^{\mu-1} (\Psi(s)-\Psi(0))^{1-\xi} g(s)\,ds\right\vert\\
&\leq 
\left\vert   \frac{\left( \Psi \left( t_2\right) -\Psi \left( 0\right) \right)
^{1-\xi }\left\| g\right\|_{C_{1-\xi ;\, \Psi }\left( J,\mathbb{R}\right) }}{\Gamma \left( \mu\right) }\int_{0}^{t_2}\Psi'(s)(\Psi(t_2)-\Psi(s))^{\mu-1} \,ds\right.\\
&\left.\qquad-\frac{\left( \Psi \left( t_1\right) -\Psi \left( 0\right) \right)
^{1-\xi }\left\| g\right\|_{C_{1-\xi ;\, \Psi }\left( J,\mathbb{R}\right) } }{\Gamma \left( \mu\right) }\int_{0}^{t_1}\Psi'(s)(\Psi(t_1)-\Psi(s))^{\mu-1}\,ds\right\vert\\
&= \frac{\left\| g\right\|_{C_{1-\xi ;\, \Psi }\left( J,\mathbb{R}\right) }}{\Gamma \left( \mu+1\right) }\left\lbrace  (\Psi(t_2)-\Psi(0))^{\mu+1-\xi} - (\Psi(t_1)-\Psi(0))^{\mu+1-\xi} \right\rbrace.
\end{align*}

By the  continuity of  $\Psi$, from the above inequality it follows that $$
\text{if}\left|t_1-t_2 \right|\to 0 ~\text{then}\left| \left( \Psi \left( t_2\right) -\Psi \left( 0\right) \right)
^{1-\xi }Fy(t_2)-\left( \Psi \left( t_1\right) -\Psi \left( 0\right) \right) ^{1-\xi }Fy(t_1)\right|\to 0.
$$ 

From the parts (ii) and (iii), it follows that $F(S)$ is uniformly bounded and equicontinous set in $X$. Then by Arzel$\acute{a}$-Ascoli theorem, $F(S)$ is relatively compact. We have proved that, $F:S\rightarrow X$ is a compact operator. Since $F:S\rightarrow X$ is the continuous and compact operator, it is completely continuous.

\textbf{Step 3:} For $y\in X,$~$y=Ey\,Fx+Gy \implies y\in S,~ \text{for all} ~x\in S$.

Let any $y\in X$ and $x\in S$   such that $y=Ey\,Fx+Gy$. Using the hypothesis {\bf (H2)} and bounds  of $u$ and $w$, for any $t\in J$, we have
\begin{align*}
&\left\vert \left( \Psi \left( t\right) -\Psi \left( 0\right) \right)
^{1-\xi }y(t) \right\vert\\
&=\left\vert \left( \Psi \left( t\right) -\Psi \left( 0\right) \right)
^{1-\xi }\left[ Ey(t)\,Fx(t) + Gy(t)\right] \right\vert\\
&=\left\vert \left( \Psi \left( t\right) -\Psi \left( 0\right) \right)
^{1-\xi }\times\right.\\
&\left.~~\left[ u(t,y(t))\left\lbrace  \frac{y_{0}}{u(0,x(0+))}\left( \Psi \left( t\right) -\Psi \left( 0\right) \right)^{\xi-1 }+\mathcal{I}_{0^+}^{\mu\,;\, \Psi}v\left(t, x(t), k\, \mathcal{I} ^{\mu\,;\, \Psi}_{0^+}x(t)\right)\right\rbrace+ w(t,y(t))\right]  \right\vert\\
&=\left\vert u(t,y(t))\left\lbrace  \frac{y_{0}}{u(0,x(0+))}+\left( \Psi \left( t\right) -\Psi \left( 0\right) \right)
^{1-\xi } \mathcal{I}_{0^+}^{\mu\,;\, \Psi}v\left(t, x(t), k\, \mathcal{I} ^{\mu\,;\, \Psi}_{0^+}x(t)\right)\right\rbrace\right.\\
&\left.~~+ \left( \Psi \left( t\right) -\Psi \left( 0\right) \right)
^{1-\xi } w(t,y(t)) \right\vert\\
&\leq\left\vert u(t,y(t))\right\vert \left\lbrace \left\vert  \frac{y_{0}}{u(0,x(0+))}\right\vert+\frac{ \left( \Psi \left( t\right) -\Psi \left(0\right) \right)
^{1-\xi }}{\Gamma \left( \mu
\right) }\int_{0}^{t}\Psi'(s)(\Psi(t)-\Psi(s))^{\mu-1} \left\vert v\left(s, x(s), k\, \mathcal{I} ^{\mu\,;\, \Psi}_{0^+}x(s)\right)\right\vert\,ds\right\rbrace\\
&\quad+ \left( \Psi \left( t\right) -\Psi \left( 0\right) \right)
^{1-\xi } \left\vert w(t,y(t))\right\vert \\
&\leq K_1 \left\lbrace \left\vert  \frac{y_{0}}{u(0,x(0+))}\right\vert+\frac{ \left( \Psi \left( t\right) -\Psi \left( 0\right) \right)
^{1-\xi }}{\Gamma \left( \mu
\right) }\int_{0}^{t}\Psi'(s)(\Psi(t)-\Psi(s))^{\mu-1} \left( \Psi \left( s\right) -\Psi \left( 0\right) \right) ^{1-\xi } g(s)\,ds\right\rbrace\\
&\quad+ \left( \Psi \left( t\right) -\Psi \left( 0\right) \right)
^{1-\xi } K_2 \\
&\leq K_1 \left\lbrace \left\vert  \frac{y_{0}}{u(0,x(0+))}\right\vert+\frac{ \left( \Psi \left( T\right) -\Psi \left( 0\right) \right)
^{\mu+1-\xi }\,}{\Gamma \left( \mu+1
\right) }\left\| g\right\|_{C_{1-\xi ;\, \Psi }\left( J,\mathbb{R}\right) }\right\rbrace+ \left( \Psi \left( T\right) -\Psi \left( 0\right) \right)
^{1-\xi } K_2 . 
\end{align*}

This gives
\begin{align*}
\left\Vert y\right\Vert _{C_{1-\xi ;\,\Psi }\left( J,\mathbb{R}\right)  }&\leq K_1 \left\lbrace \left\vert  \frac{y_{0}}{u(0,x(0+))}\right\vert+\frac{ \left( \Psi \left( T\right) -\Psi \left( 0\right) \right)
^{\mu+1-\xi }\,\left\| g\right\|_{C_{1-\xi ;\, \Psi }\left( J,\mathbb{R}\right) }}{\Gamma \left( \mu+1
\right) }\right\rbrace\\
&\quad+ \left( \Psi \left( T\right) -\Psi \left( 0\right) \right)
^{1-\xi } K_2 =R. 
\end{align*}

This implies, $y\in S$.

\textbf{Step 4:} To prove $4\,\sigma M + \delta<1,$ where $M=\sup\left\lbrace \left\|Fy \right\|_{C_{1-\xi ;\,\Psi }\left( J,\mathbb{R}\right)}: y\in  S \right\rbrace $.

From inequality \eqref{619}, we have
\begin{align*}
M&=\sup\left\lbrace \left\|Fy \right\|_{C_{1-\xi ;\,\Psi }\left( J,\mathbb{R}\right)}: y\in  S \right\rbrace\\
&\leq \left\vert \frac{y_0}{u(0,y(0+))}\right\vert +\frac{\left( \Psi \left( T\right) -\Psi \left( 0\right) \right)^{\mu+1-\xi }}{\Gamma \left( \mu+1   \right) } \left\| g\right\|_{C_{1-\xi ;\, \Psi }\left( J,\mathbb{R}\right) }.
\end{align*}

Now, using the condition \eqref{616}, we have 
$$
4\,\sigma M + \delta\leq 4\,\sigma \left\lbrace \left\vert \frac{y_0}{u(0,y(0+))}\right\vert +\frac{\left( \Psi \left( T\right) -\Psi \left( 0\right) \right)^{\mu+1-\xi }}{\Gamma \left( \mu+1   \right) } \left\| g\right\|_{C_{1-\xi ;\, \Psi }\left( J,\mathbb{R}\right) }\right\rbrace  + \delta <1.
$$
From {\bf Steps 1} to {\bf 4}, it follows that all the conditions of Lemma \ref{hyb2} are fulfilled. Consequently,  by applying Lemma \ref{hyb2}, the operator  $T$ has a coupled solution in  $\tilde{S}=S\times S$. Hence, the  coupled system of hybrid {\rm FDEs} \eqref{eqq611}-\eqref{eqq612}  has a solution in $ C_{1-\xi ;\, \Psi }\left( J,\mathbb{R}\right)\times C_{1-\xi ;\, \Psi }\left( J,\mathbb{R}\right) $.
\end{proof}
\section{{\rm BVPs} for Coupled system of Hybrid {\rm FDEs}}
In this section, we are concerned with the  {\rm BVPs} for coupled system of $\Psi$-Hilfer hybrid {\rm FDEs} {\rm \eqref{eq663}-\eqref{eq664}}. Consider  the product space $E =X\times X,$~$X=C_{1-\xi ;\Psi }\left( J,\mathbb{R}\right)  $ with 
\begin{enumerate}

\item [(i)] vector addition: $(p, q)(t) + (\bar{p}, \bar{q})(t)=\left(  p(t)+\bar{p}(t), q(t)+\bar{q}(t)\right),
$
\item [(ii)]   scalar multiplication: 
$k\,(p, q)(t)=\left(  k\,p(t), k\,q(t)\right)$,
\end{enumerate} 
where,  $t\in J$, $p,\, q, \,\bar{p},\, \bar{q}\in X$ and $k\in \R$.
Then, $E $  is  a Banach algebra endowed with the norm
\begin{equation}\label{eqq664}
\left\Vert(p, q)\right\Vert _E =
\left\Vert p\right\Vert _{C_{1-\xi ;\Psi }\left( J,\mathbb{R}\right)  }+\left\Vert q \right\Vert _{C_{1-\xi ;\Psi }\left( J,\mathbb{R}\right)  }
\end{equation}
and the vector multiplication defined by
$$
(p, q)(t) \cdot (\bar{p}, \bar{q})(t)=\left(  p(t)\bar{p}(t), q(t)\bar{q}(t)\right),\, \text{for any}\,  (p, q), (\bar{p}, \bar{q})\in  E~ \text{and }\, t \in J.
$$
  
\begin{theorem} The {\rm BVP} for $\Psi$-Hilfer hybrid {\rm FDEs} 
 \begin{align}
   & ^H \mathcal{D}^{\mu,\,\nu\,;\, \Psi}_{0^+}\left[ \frac{y(t)-w_1(t, y(t), x(t))}{ u_1(t, y(t), x(t))}\right] 
      = v_1\left(t, y(t), x(t)\right),~a.e. ~t \in  (0,\,T],  ~\label{eqq661}\\
      & a\,\lim\limits_{t\to 0+}\left( \Psi \left( t\right) -\Psi \left( 0\right) \right)^{1-\xi }y(t)+b\,\lim\limits_{t\to T} \left( \Psi \left( t\right) -\Psi \left( 0\right) \right)^{1-\xi }y(t)=y_{0}\in\R,\label{eqq663}
    \end{align}
is   equivalent to the fractional IE
\begin{align} \label{eqq664a}
  y(t)&=w_1(t, y(t), x(t))\nonumber\\
  &\quad+ u_1(t, y(t), x(t))\left[ \left( \Psi \left( t\right) -\Psi \left( 0\right) \right)^{\xi-1 } \Omega_1+  \mathcal{I}^{\mu;\, \Psi}_{0^+} v_1(t, y(t), x(t))\right],\,t\in(0, T],
\end{align}
 where
\begin{small}
$$ \Omega_1=\frac{y_0- b\, \left( \Psi \left( T\right) -\Psi \left( 0\right) \right)^{1-\xi } \left(   w_1(T, y(T), x(T))+\,  u_1(T, y(T), x(T)) \,\mathcal{I}^{\mu;\, \Psi}_{0^+} v_1(T, y(T), x(T))\right)}{a\,u_1(0, y(0+), x(0+))+b\, u_1(T, y(T), x(T))}.
 $$
\end{small}
\end{theorem}

\begin{proof} Let $y\in C_{1-\xi;\Psi}(J, \R)$ is a solution of the {\rm BVP} for $\Psi$-Hilfer hybrid {\rm FDEs} \eqref{eqq661}-\eqref{eqq663}. Taking $\mathcal{I}^{\mu;\, \Psi}_{0^+}$ on both sides of Eq.\eqref{eqq661} and using Lemma \ref{teo1} (i), we get
  \begin{small}
$$
\frac{y(t)-w_1(t, y(t), x(t))}{ u_1(t, y(t), x(t))}-\frac{\left( \Psi \left( t\right) -\Psi \left( 0\right) \right)^{\xi-1 }}{\Gamma(\xi)}\left[\mathcal{I}^{1-\xi;\, \Psi}_{0^+}\frac{y(t)-w_1(t, y(t), x(t))}{ u_1(t, y(t), x(t))} \right]_{t=0}=   \mathcal{I}^{\mu;\, \Psi}_{0^+} v_1(t, y(t), x(t)).
$$ 
\end{small}

Let $C^*=\left[\mathcal{I}^{1-\xi;\, \Psi}_{0^+}\dfrac{y(t)-w_1(t, y(t), x(t))}{ u_1(t, y(t), x(t))} \right]_{t=0}$. Thus, we have 
$$\frac{y(t)-w_1(t, y(t), x(t))}{ u_1(t, y(t), x(t))}=  \frac{\left( \Psi \left( t\right) -\Psi \left( 0\right) \right)^{\xi-1 }}{\Gamma(\xi)}\, C^*+ \mathcal{I}^{\mu;\, \Psi}_{0^+} v_1(t, y(t), x(t)).$$
 
Therefore,
\begin{equation}\label{eqq665}
y(t)=  w_1(t, y(t), x(t))+ u_1(t, y(t), x(t))\left[ \frac{\left( \Psi \left( t\right) -\Psi \left( 0\right) \right)^{\xi-1 }}{\Gamma(\xi)}\, C^*+ \mathcal{I}^{\mu;\, \Psi}_{0^+} v_1(t, y(t), x(t)) \right].
\end{equation}

Now, we find the value of $C^*$ using condition \eqref{eqq663}. Multiplying by $\left( \Psi \left( t\right) -\Psi \left( 0\right) \right)^{1-\xi }$ on both sides of Eq.\eqref{eqq665}, we get
\begin{align}\label{eqq666}
\left( \Psi \left( t\right) -\Psi \left( 0\right) \right)^{1-\xi }  y(t)&= \left( \Psi \left( t\right) -\Psi \left( 0\right) \right)^{1-\xi }  w_1(t, y(t), x(t))\nonumber\\
&+ u_1(t, y(t), x(t))\left[ \frac{C^*}{\Gamma(\xi)}\, + \left( \Psi \left( t\right) -\Psi \left( 0\right) \right)^{1-\xi }\,\mathcal{I}^{\mu;\, \Psi}_{0^+} v_1(t, y(t), x(t)) \right].
\end{align}

Taking limit as $t\to 0+$ in Eq.\eqref{eqq666}, we obtain
\begin{equation}\label{eqq667}
\lim\limits_{t\to 0+}\left( \Psi \left( t\right) -\Psi \left( 0\right) \right)^{1-\xi }    y(t) =    \frac{u_1(0, y(0+), x(0+))}{\Gamma(\xi)}\, C^*.
\end{equation}  

Further, taking limit as $t\to T$ in Eq.\eqref{eqq666}, we obtain
\begin{align}\label{eqq668}
  &\lim\limits_{t\to T} \left( \Psi \left( t\right) -\Psi \left( 0\right) \right)^{1-\xi }  y(t)\nonumber\\
   &= \left( \Psi \left( T\right) -\Psi \left( 0\right) \right)^{1-\xi }  w_1(T, y(T), x(T))\nonumber\\
    &\quad+ u_1(T, y(T), x(T))\left[ \frac{C^*}{\Gamma(\xi)}\, + \left( \Psi \left( T\right) -\Psi \left( 0\right) \right)^{1-\xi }\,\mathcal{I}^{\mu;\, \Psi}_{0^+} v_1(T, y(T), x(T)) \right].
\end{align}

Using Eqs.\eqref{eqq667}-\eqref{eqq668} in the Eq.\eqref{eqq663}, we get
  \begin{align*}
   y_0&= a\,\frac{u_1(0, y(0+), x(0+))}{\Gamma(\xi)}\, C^* +b\, \left( \Psi \left( T\right) -\Psi \left( 0\right) \right)^{1-\xi }  w_1(T, y(T), x(T))\nonumber\\
    &\quad+b\,  u_1(T, y(T), x(T))\left[ \frac{C^*}{\Gamma(\xi)}\, + \left( \Psi \left( T\right) -\Psi \left( 0\right) \right)^{1-\xi }\,\mathcal{I}^{\mu;\, \Psi}_{0^+} v_1(T, y(T), x(T)) \right]\\
    &=C^*\left[ a\,\frac{u_1(0, y(0+), x(0+))}{\Gamma(\xi)}
    + b\,\frac{u_1(T, y(T), x(T))}{\Gamma(\xi)}\right] \\
    &+b\, \left( \Psi \left( T\right) -\Psi \left( 0\right) \right)^{1-\xi } \left(  w_1(T, y(T), x(T))+
          u_1(T, y(T), x(T)) \,\mathcal{I}^{\mu;\, \Psi}_{0^+} v_1\left( T, y(T), x(T)\right) \right) .
        \end{align*}   

This gives,
  \begin{align*}
  C^*&=\frac{\Gamma(\xi)}{a\,u_1(0, y(0+), x(0+))+b\,u_1(T, y(T), x(T))}\times\\
  &~\left[ y_0-b\, \left( \Psi \left( T\right) -\Psi \left( 0\right) \right)^{1-\xi } \left(  w_1(T, y(T), x(T))+
            u_1(T, y(T), x(T)) \,\mathcal{I}^{\mu;\, \Psi}_{0^+} v_1\left( T, y(T), x(T)\right) \right)\right].
  \end{align*}

Putting value of $C^*$ in the Eq.\eqref{eqq665}, we obtain
   \begin{align*}
  &y(t)\\
  &=  w_1(t, y(t), x(t))+ u_1(t, y(t), x(t))\left\lbrace \left( \Psi \left( t\right) -\Psi \left( 0\right) \right)^{\xi-1 }\times\right.\\
  &\left. \frac{\left[ y_0 - b\, \left( \Psi \left( T\right) -\Psi \left( 0\right) \right)^{1-\xi } \left(   w_1(T, y(T), x(T))+\,  u_1(T, y(T), x(T)) \,\mathcal{I}^{\mu;\, \Psi}_{0^+} v_1(T, y(T), x(T))\right)  \right] }{a\,u_1(0, y(0+), x(0+))+b\,u_1(T, y(T), x(T))}\right.\\
  &\left.+ \mathcal{I}^{\mu;\, \Psi}_{0^+} v_1(t, y(t), x(t))\right\rbrace\\
  &=  w_1(t, y(t), x(t))+ u_1(t, y(t), x(t))\left\lbrace \left( \Psi \left( t\right) -\Psi \left( 0\right) \right)^{\xi-1 }\,\Omega_1+ \mathcal{I}^{\mu;\, \Psi}_{0^+} v_1(t, y(t), x(t))\right\rbrace,~ t\in(0,T],
    \end{align*}
which is the fractional IE \eqref{eqq664a}.

Conversely, let $y\in C_{1-\xi;\Psi}(J, \R)$ be a solution of the  Volterra IE \eqref{eqq664a}. The Eq.\eqref{eqq664a} can be rewritten as 
\begin{align*}
\frac{y(t)-  w_1(t, y(t), x(t))}{u_1(t, y(t), x(t))}=\left( \Psi \left( t\right) -\Psi \left( 0\right) \right)^{\xi-1 }\,\Omega_1+ \mathcal{I}^{\mu;\, \Psi}_{0^+} v_1(t, y(t), x(t)).
\end{align*}

Taking $\Psi$-Hilfer  fractional derivative $^H\mathcal{D}^{\mu,\,\nu\,;\, \Psi}_{0^+}$ on both sides and using Lemma \ref{lema2} (iii) and  Lemma \ref{teo1} (i), we obtain
$$ ^H \mathcal{D}^{\mu,\,\nu\,;\, \Psi}_{0^+}\left[ \frac{y(t)-w_1(t, y(t), x(t))}{ u_1(t, y(t), x(t))}\right] 
= v_1\left(t, y(t), x(t)\right),~a.e. ~t \in  (0,\,T],
$$
which is Eq.\eqref{eqq661}.  Multiplying IE \eqref{eqq664a} by $\left( \Psi \left( t\right) -\Psi \left( 0\right) \right)^{1-\xi }$, we obtain
\begin{align}\label{eqq669}
&\left( \Psi \left( t\right) -\Psi \left( 0\right) \right)^{1-\xi } y(t)\nonumber\\
&=\left( \Psi \left( t\right) -\Psi \left( 0\right) \right)^{1-\xi } w_1(t, y(t), x(t))\nonumber\\
&\quad+ u_1(t, y(t), x(t))\left\lbrace \,\Omega_1+\left( \Psi \left( t\right) -\Psi \left( 0\right) \right)^{1-\xi } \mathcal{I}^{\mu;\, \Psi}_{0^+} v_1(t, y(t), x(t))\right\rbrace,~ t\in J.
\end{align}

Taking limit as $t\to 0+$, from above Eq.\eqref{eqq669}, we obtain
\begin{equation}\label{eqq670}
\lim\limits_{t\to 0+}\left( \Psi \left( t\right) -\Psi \left( 0\right) \right)^{1-\xi } y(t)=u_1(0, y(0+), x(0+)) \,\Omega_1.
\end{equation}

Further, taking limit as $t\to T$, from  Eq.\eqref{eqq669}, we obtain
\begin{align}\label{eqq671}
&\lim\limits_{t\to T}\left( \Psi \left( t\right) -\Psi \left( 0\right) \right)^{1-\xi } y(t)\nonumber\\
&=\left( \Psi \left( T\right) -\Psi \left( 0\right) \right)^{1-\xi } w_1(T, y(T), x(T))\nonumber\\
&\quad +u_1(T, y(T), x(T))\left\lbrace \,\Omega_1+\left( \Psi \left( T\right) -\Psi \left( 0\right) \right)^{1-\xi} \mathcal{I}^{\mu;\, \Psi}_{0^+} v_1(T, y(T), x(T))\right\rbrace.
\end{align}

Using the Eqs.\eqref{eqq670}-\eqref{eqq671} and the value of $\Omega_1$, consider
\begin{align*}
&a\,\lim\limits_{t\to 0+}\left( \Psi \left( t\right) -\Psi \left( 0\right) \right)^{1-\xi } y(t)+b\,\lim\limits_{t\to T}\left( \Psi \left( t\right) -\Psi \left( 0\right) \right)^{1-\xi } y(t)\\
&= a\,u_1(0, y(0+), x(0+)) \,\Omega_1+b\,\left( \Psi \left( T\right) -\Psi \left( 0\right) \right)^{1-\xi } w_1(T, y(T), x(T))\nonumber\\
&\quad+ b\,u_1(T, y(T), x(T))\left\lbrace \,\Omega_1+\left( \Psi \left( T\right) -\Psi \left( 0\right) \right)^{1-\xi } \mathcal{I}^{\mu;\, \Psi}_{0^+} v_1(T, y(T), x(T))\right\rbrace\\
&=\left[ a\,u_1(0, y(0+), x(0+)) +b\,u_1(T, y(T), x(T))\right] \,\Omega_1+b\,\left( \Psi \left( T\right) -\Psi \left( 0\right) \right)^{1-\xi } w_1(T, y(T), x(T))\\
&\quad+ b\,u_1(T, y(T), x(T))\,\left( \Psi \left( T\right) -\Psi \left( 0\right) \right)^{1-\xi} \mathcal{I}^{\mu;\, \Psi}_{0^+} v_1(T, y(T), x(T))\\
&=y_0,
\end{align*}
which is the condition \eqref{eqq663}. This proves, $y\in C_{1-\xi;\Psi}(J, \R)$  is a solution of the {\rm BVP} for {\rm FDEs} involving $\Psi$-Hilfer  fractional derivative  \eqref{eqq661}-\eqref{eqq663}.
\end{proof}
  
To prove the existence of solution  to the {\rm BVPs} for coupled system of $\Psi$-Hilfer hybrid {\rm FDEs} {\rm \eqref{eq663}-\eqref{eq664}}, we need the following hypotheses  on $u_i$, $v_i$  and $w_i(i=1,2)$.

 \begin{enumerate}[topsep=0pt,itemsep=-1ex,partopsep=1ex,parsep=1ex]
  \item [{\bf (H3)}] The functions $u_i\in C(J \times \R  \times \R \,, \R\setminus\{0\}) (i=1,2)$,    $w_i\in C(J \times \R  \times \R \,, \R)(i=1,2)$  are bounded  and there exists constants $\sigma_i, \delta_i>0(i=1,2)$ such that for all $p, q, \bar{p}, \bar{q}\in \R,\,i=1,2$ and $t\in J=[0, T]$, we have 
   $$
    \left| u_i(t,p, q) - u_i(t,\bar{p}, \bar{q})\right| \leq \sigma_i\left(  \left| p-\bar{p}\right|+\left| q-\bar{q}\right|\right) 
    $$
    and 
    $$
     \left| w_i(t,p, q) - w_i(t,\bar{p}, \bar{q})\right| \leq \delta_i\left(  \left| p-\bar{p}\right|+\left| q-\bar{q}\right|\right).
        $$
  
  \item [{\bf (H4)}] The functions $v_i\in C(J \times \R\times \R  \,, \R)(i=1,2) $  and there exists a functions $ g_i\in C_{1-\xi;\Psi}\left( J,\mathbb{R}\right)  $ such that
   $$
    \left| v_i(t, p, q)\right| \leq \left( \Psi \left( t\right) -\Psi \left( 0\right) \right)^{1-\xi } g_i(t),\,~a.e.\,\, t\in J~ \text{and}~  \,p, q\in \R.
   $$
 
   \end{enumerate}
\begin{theorem}\label{tha64.2} Assume that the  hypotheses {\normalfont{\bf (H3)}-{\bf (H4)}} hold. Then, the {\rm BVPs} for coupled system of $\Psi$-Hilfer hybrid {\rm FDEs} {\rm \eqref{eq663}-\eqref{eq664}} has a solution $(y, x)\in E$ provided
\begin{small}
\begin{equation}\label{eqq672}
\left(\sum_{i=1}^{2} \sigma_i\right) \left[ \sum_{i=1}^{2}\left\vert \Omega_i\right|+\frac{\left( \Psi \left( T\right) -\Psi \left( 0\right) \right)^{\mu+1-\xi }}{\Gamma \left( \mu+1   \right) } \left(  \sum_{i=1}^{2}\left\| g_i\right\|_{C_{1-\xi ;\, \Psi }\left( J,\mathbb{R}\right) }\right)  \right]  +\sum_{i=1}^{2} \delta_i<1,
\end{equation}
 \end{small}
where
\begin{small}
\begin{equation}\label{ie61}
  \Omega_1=\frac{y_0- b\, \left( \Psi \left( T\right) -\Psi \left( 0\right) \right)^{1-\xi } \left(   w_1(T, y(T), x(T))+\,  u_1(T, y(T), x(T)) \,\mathcal{I}^{\mu;\, \Psi}_{0^+} v_1(T, y(T), x(T))\right)}{a\,u_1(0, y(0+), x(0+))+b\, u_1(T, y(T), x(T))}
\end{equation}
   \end{small}
and 
\begin{small}
\begin{equation}\label{ie62}
      \Omega_2=\frac{y_0- b\, \left( \Psi \left( T\right) -\Psi \left( 0\right) \right)^{1-\xi } \left(   w_2(T, y(T), x(T))+\,  u_2(T, y(T), x(T)) \,\mathcal{I}^{\mu;\, \Psi}_{0^+} v_2(T, y(T), x(T))\right)}{a\,u_2(0, y(0+), x(0+))+b\, u_2(T, y(T), x(T))}.
\end{equation}
\end{small}
\end{theorem}
\begin{proof} Define, $$ S^*=\{(y, x)\in X\times X: \left\Vert (y, x)\right\Vert _E\leq R^*\},$$
where 
\begin{align*}
R^*&= M_1  \left\vert  \bar{ \Omega}_1\right\vert + M_2  \left\vert  \bar{ \Omega}_2\right\vert + \left( \Psi \left( T\right) -\Psi \left( 0\right) \right)^{1-\xi } \left[ N_1+N_2\right] \\
   &+\frac{ \left( \Psi \left( T\right) -\Psi \left( 0\right) \right)^{\mu+1-\xi }\,}{\Gamma \left( \mu+1\right) }\left[ M_1\,\left\| g_1\right\|_{C_{1-\xi ;\, \Psi }\left( J,\mathbb{R}\right) } + M_2\,\left\| g_2\right\|_{C_{1-\xi ;\, \Psi }\left( J,\mathbb{R}\right) }\right] 
\end{align*}  
and $M_i>0$ and  $N_i>0~(i=1,2)$ are the constants such that   $\left| u_i(t, \cdot, \cdot)\right| <M_i$ and $\left| w_i(t, \cdot, \cdot)\right| <N_i,$ for all $t\in J$. Clearly, $S^*$ is non-empty, closed, convex and bounded subset of $E=X\times X$.

If $(y, x)\in  S^*\subseteq X\times X$ is a solution of the coupled system of nonlinear $\Psi$-Hilfer hybrid {\rm FDEs} \eqref{eq663}-\eqref{eq664}, then it is a solution of the coupled system of fractional IEs   
 \begin{small}
\begin{align}\label{eqq673}
\begin{cases}
 & y(t)= u_1(t, y(t), x(t))\left[ \left( \Psi \left( t\right) -\Psi \left( 0\right) \right)^{\xi-1 } \Omega_1+  \mathcal{I}^{\mu;\, \Psi}_{0^+} v_1(t, y(t), x(t))\right]+w_1(t, y(t), x(t)),~t\in(0,T]\\
 &  x(t)= u_2(t, y(t), x(t))\left[ \left( \Psi \left( t\right) -\Psi \left( 0\right) \right)^{\xi-1 } \Omega_2+  \mathcal{I}^{\mu;\, \Psi}_{0^+} v_2(t, y(t), x(t))\right]+w_2(t, y(t), x(t)),~t\in(0,T],
\end{cases}
 \end{align}  
\end{small}
where $\Omega_1$  and $\Omega_2$ are defined in Eqs.\eqref{ie61} and \eqref{ie62} respectively.

 For $ i=1,2,$ define the operators $A_i:E\to X$, $B_i:S^*\to X$ and $C_i:E\to X$  by
\begin{align*}
 & A_1\left(y, x\right)(t) = u_1\left( t, y(t), x(t)\right), \,t\in J;\\
  & A_2\left(y, x\right)(t) = u_2\left( t, y(t), x(t)\right), \,t\in J;\\
  &B_1\left( y, x\right)(t) = \left( \Psi \left( t\right) -\Psi \left( 0\right) \right)^{\xi-1 } \Omega_1+  \mathcal{I}^{\mu;\, \Psi}_{0^+} v_1(t, y(t), x(t)),~t\in(0,T];\\
    &B_2\left( y, x\right)(t) = \left( \Psi \left( t\right) -\Psi \left( 0\right) \right)^{\xi-1 } \Omega_2 +  \mathcal{I}^{\mu;\, \Psi}_{0^+} v_2(t, y(t), x(t)),~t\in(0,T];\\
& C_1\left(y, x\right)(t) = w_1\left( t, y(t), x(t)\right), \,t\in J;\\
  & C_2\left(y, x\right)(t) = w_2\left( t, y(t), x(t)\right), \,t\in J.
\end{align*}
Then, the coupled system of  hybrid IEs in Eq.\eqref{eqq673} transformed into
\begin{align}\label{ak}
\begin{cases}
&A_1\left(y, x\right)(t)\, B_1\left( y, x\right)(t) + C_1\left(y, x\right)(t)= y(t),\,~t\in(0,T],\\
&A_2\left(y, x\right)(t)\, B_2\left( y, x\right)(t) + C_2\left(y, x\right)(t)= x(t),\,~t\in(0,T].
\end{cases}
\end{align}
Consider the operators, $A = (A_1, A_2) : E \rightarrow E$, $ B = (B_1, B_2) : S^* \rightarrow E$ and $ C = (C_1, C_2) : E \rightarrow E$. Then, the operator equations in \eqref{ak} can be written as
\begin{equation}\label{ak1}
A(y, x)(t)\,B(y, x)(t) + C(y, x)(t) = (y, x)(t),\, (y, x)\in E~ \text{and}~ t\in J.
\end{equation}
We prove  that the operators $A$, $B$ and $C$ satisfies all the conditions of Lemma \ref{hyb3}.  The proof is given in the following  series of steps.\\ \\
\noindent\textbf{Step 1:} $  A = (A_1, A_2) : E \rightarrow E$ and $ C = (C_1, C_2) : E \rightarrow E$ are Lipschitz operators.

For  any $(y, x), (\bar{y}, \bar{x})\in E$ and  $t\in J$, we obtain
\begin{align}\label{h1}
&\left\Vert  A(y, x)-A (\bar{y}, \bar{x}) \right\Vert _E\nonumber\\
&=\left\Vert  \left( A_1(y, x), A_2(y, x)\right) -\left( A_1 (\bar{y}, \bar{x}), A_2 (\bar{y}, \bar{x})\right)  \right\Vert  _E\nonumber\\
&=\left\Vert  \left( A_1(y, x) - A_1 (\bar{y}, \bar{x})\right),  \left( A_2(y, x) - A_2 (\bar{y}, \bar{x})\right)  \right\Vert  _E\nonumber\\
&=\left\Vert  A_1(y, x)-A_1 (\bar{y}, \bar{x}) \right\Vert _{C_{1-\xi ;\,\Psi }\left( J,\mathbb{R}\right)  }+\left\Vert  A_2(y, x)-A_2 (\bar{y}, \bar{x}) \right\Vert _{C_{1-\xi ;\,\Psi }\left( J,\mathbb{R}\right)  }.
\end{align}

Now,  using the hypothesis {\bf (H3)}, we obtain
\begin{align*}
&\left| \left( \Psi \left( t\right) -\Psi \left( 0\right) \right)^{1-\xi} \left(  A_1(y, x)(t)-A_1 (\bar{y}, \bar{x})(t) \right)   \right|\\
 &=\left| \left( \Psi \left( t\right) -\Psi \left( 0\right) \right)^{1-\xi }\left( u_1\left( t, y(t), x(t)\right)-u_1\left( t, \bar{y}(t), \bar{x}(t)\right) \right) \right|\nonumber \\
 &\leq \sigma_1 \left( \Psi \left( t\right) -\Psi \left( 0\right) \right)^{1-\xi } \left[ \left| y(t)-\bar{y}(t)\right| +\left|   x(t)-\bar{x}(t)\right|\right] \nonumber \\
 &\leq \sigma_1 \left[ \left\Vert y-\bar{y}\right\Vert _{C_{1-\xi ;\,\Psi }\left( J,\mathbb{R}\right)  }+ \left\Vert x-\bar{x}\right\Vert _{C_{1-\xi ;\,\Psi }\left( J,\mathbb{R}\right)  }\right].
\end{align*}

This gives,
\begin{equation}\label{h2}
\left\Vert  A_1(y, x)-A_1 (\bar{y}, \bar{x}) \right\Vert _{C_{1-\xi ;\,\Psi }\left( J,\mathbb{R}\right)  }\leq \sigma_1 \left[ \left\Vert y-\bar{y}\right\Vert _{C_{1-\xi ;\,\Psi }\left( J,\mathbb{R}\right)  }+ \left\Vert x-\bar{x}\right\Vert _{C_{1-\xi ;\,\Psi }\left( J,\mathbb{R}\right)  }\right]. 
\end{equation}
Similarly, we have 
\begin{equation}\label{h3}
\left\Vert  A_2(y, x)-A_2 (\bar{y}, \bar{x}) \right\Vert _{C_{1-\xi ;\,\Psi }\left( J,\mathbb{R}\right)  }\leq \sigma_2 \left[ \left\Vert y-\bar{y}\right\Vert _{C_{1-\xi ;\,\Psi }\left( J,\mathbb{R}\right)  }+ \left\Vert x-\bar{x}\right\Vert _{C_{1-\xi ;\,\Psi }\left( J,\mathbb{R}\right)  }\right]. 
\end{equation}

Using the inequalities \eqref{h2} and \eqref{h3}, from Eq.\eqref{h1}, we have
\begin{equation*}
\left\Vert  A(y, x)-A (\bar{y}, \bar{x}) \right\Vert _E\leq \left( \sigma_1+\sigma_2 \right)  \left[ \left\Vert y-\bar{y}\right\Vert _{C_{1-\xi ;\,\Psi }\left( J,\mathbb{R}\right)  }+ \left\Vert x-\bar{x}\right\Vert _{C_{1-\xi ;\,\Psi }\left( J,\mathbb{R}\right)  }\right].
\end{equation*}

Therefore, $A$ is Lipschitz operator with  Lipschitz constant $K=\sigma_1+\sigma_2$. On the similar line, it is easy to prove that  $C$ is Lipschitz operator. Let  $L=\delta_1+\delta_2$ is the  Lipschitz constant corresponding to the operator $C$.
\\ \\
\noindent\textbf{Step 2:} $B = (B_1, B_2) : S^* \rightarrow E$  is completely continuous.

(a) $B = (B_1, B_2) : S^* \rightarrow E$ is continuous.

Let  $\left( y_n, x_n\right)  $ be any sequence of points in $S^*$ such that $\left( y_n, x_n\right) \rightarrow \left( y , x\right)$ as $n\rightarrow \infty$  in $S^*$. We prove that  $B\left( y_n, x_n\right) \rightarrow B\left( y , x\right) $ as $n\rightarrow \infty$  in $E$.  

Consider,
\begin{align*}
&\left\Vert B_1\left( y_n, x_n\right) - B_1\left( y , x\right) \right\Vert _{C_{1-\xi ;\,\Psi }\left( J,\mathbb{R}\right)  }\nonumber\\
&=\underset{t\in J 
}{\max }\left\vert \left( \Psi \left( t\right) -\Psi \left( 0\right) \right)
^{1-\xi }\left(   B_1\left( y_n, x_n\right)(t) - B_1\left( y , x\right) (t)\right)   \right\vert\nonumber\\
&\leq  \underset{t\in J }{\max }\frac{\left( \Psi \left( t\right) -\Psi \left( 0\right) \right)
^{1-\xi }}{\Gamma \left( \mu\right) } \int_{0}^{t}\Psi'(s)(\Psi(t)-\Psi(s))^{\mu-1} \left\vert  v_1\left( s, y_n(s), x_n(s)\right) -v_1\left( s, y(s), x(s)\right)  \right\vert\,ds.
\end{align*}

By continuity of  the  function   $v_1$ and the  Lebesgue dominated convergence theorem, from the above inequality,  we obtain
$$
\left\Vert B_1\left( y_n, x_n\right) - B_1\left( y , x\right)\right\Vert _{C_{1-\xi ;\,\Psi }\left( J,\mathbb{R}\right)  }\rightarrow 0 ~\text{ as}~ n\rightarrow\infty.
$$

On the similar line one can obtain
$$
\left\Vert B_2\left( y_n, x_n\right) - B_2\left( y , x\right)\right\Vert _{C_{1-\xi ;\,\Psi }\left( J,\mathbb{R}\right)  }\rightarrow 0 ~\text{ as}~ n\rightarrow\infty.
$$

Hence, $B\left( y_n, x_n\right) = \left( B_1\left( y_n, x_n\right), B_2\left( y_n, x_n\right) \right)  $ converges to  $B\left( y, x\right) = \left( B_1\left( y, x\right), B_2\left( y, x\right) \right)  $ as $n\rightarrow \infty$. 

This proves $B: S^* \rightarrow E$ is continuous. 

(b) $B(S^*)=\left\lbrace B\left(y, x\right) : \left(y, x\right) \in S^*\right\rbrace $ is uniformly bounded.

Using hypothesis {\bf (H4)}, for any $\left(y, x\right)\in S^*$ and $t\in J$, we have
\begin{align*}
&\left\vert \left( \Psi \left( t\right) -\Psi \left( 0\right) \right)
^{1-\xi }B_1\left(y, x\right) (t) \right\vert\\
 &\leq \left\vert \Omega_1\right| +\frac{\,\left( \Psi \left( t\right) -\Psi \left( 0\right) \right)^{1-\xi }}{\Gamma \left( \mu
       \right) }\int_{0}^{t}\Psi'(s)(\Psi(t)-\Psi(s))^{\mu-1} \left\vert v_1\left(s, y(s), x(s)\right)\right\vert \,ds\\
&\leq \left\vert \Omega_1\right| +\frac{\,\left( \Psi \left( t\right) -\Psi \left( 0\right) \right)^{1-\xi }}{\Gamma \left( \mu
       \right) }\int_{0}^{t}\Psi'(s)(\Psi(t)-\Psi(s))^{\mu-1} (\Psi(s)-\Psi(0))^{1-\xi} g_1(s)\,ds\\
&\leq \left\vert \Omega_1\right| +\left\| g_1\right\|_{C_{1-\xi ;\, \Psi }\left( J,\mathbb{R}\right) }\,\left( \Psi \left( t\right) -\Psi \left( 0\right) \right)^{1-\xi }\frac{\left( \Psi \left( t\right) -\Psi \left( 0\right) \right)^{\mu }}{\Gamma \left( \mu+1     \right) }   \\    
&\leq \left\vert \Omega_1\right|  +\frac{\left( \Psi \left( T\right) -\Psi \left( 0\right) \right)^{\mu+1-\xi }}{\Gamma \left( \mu+1     \right) } \left\| g_1\right\|_{C_{1-\xi ;\, \Psi }\left( J,\mathbb{R}\right) }. 
\end{align*}

Therefore, 
\begin{equation}\label{eqq675}
\left\Vert B_1\left(y, x\right)\right\Vert _{C_{1-\xi ;\,\Psi }\left( J,\mathbb{R}\right)  }\leq \left\vert \Omega_1\right| +\frac{\left( \Psi \left( T\right) -\Psi \left( 0\right) \right)^{\mu+1-\xi }}{\Gamma \left( \mu+1   \right) } \left\| g_1\right\|_{C_{1-\xi ;\, \Psi }\left( J,\mathbb{R}\right) },~ \text{for all}\, \left(y, x\right)\in S^*.
\end{equation}

Hence, $B_1$ is  uniformly bounded  on $S^*$. On the similar line, one can obtain
 \begin{equation}\label{eqq676}
 \left\Vert B_2\left(y, x\right)\right\Vert _{C_{1-\xi ;\,\Psi }\left( J,\mathbb{R}\right)  }\leq \left\vert \Omega_2\right| +\frac{\left( \Psi \left( T\right) -\Psi \left( 0\right) \right)^{\mu+1-\xi }}{\Gamma \left( \mu+1   \right) } \left\| g_2\right\|_{C_{1-\xi ;\, \Psi }\left( J,\mathbb{R}\right) },~ \text{for all}\, \left(y, x\right)\in S^*.
 \end{equation}

This proves $B_2$ is  uniformly bounded  on $S^*$.  Hence, the operator $B$ is  uniformly bounded on $S^*$.

(c) $B(S^*)=(B_1(S^*), B_2(S^*))$ is equicontinuous.

Let any $\left(y, x\right)\in S^*$ and $t_1, t_2\in J $ with $t_1<t_2$. Then, using hypothesis {\bf (H4)}, we have
\begin{align*}
&\left\vert \left( \Psi \left( t_2\right) -\Psi \left( 0\right) \right)
^{1-\xi }B_1\left(y, x\right)(t_2)-\left( \Psi \left( t_1\right) -\Psi \left( 0\right) \right) ^{1-\xi }B_1\left(y, x\right)(t_1) \right\vert\\
&\leq 
\left\vert   \frac{\left( \Psi \left( t_2\right) -\Psi \left( 0\right) \right)
^{1-\xi } }{\Gamma \left( \mu\right) }\int_{0}^{t_2}\Psi'(s)(\Psi(t_2)-\Psi(s))^{\mu-1} \left| v_1\left(s, y(s), x(s)\right)\right| \,ds\right.\\
&\left.\qquad-\frac{\left( \Psi \left( t_1\right) -\Psi \left( 0\right) \right)
^{1-\xi } }{\Gamma \left( \mu\right) }\int_{0}^{t_1}\Psi'(s)(\Psi(t_1)-\Psi(s))^{\mu-1} \left| v_1\left(s, y(s), x(s)\right)\right|\,ds\right\vert\\
&\leq 
\left\vert   \frac{\left( \Psi \left( t_2\right) -\Psi \left( 0\right) \right)
^{1-\xi } }{\Gamma \left( \mu\right) }\int_{0}^{t_2}\Psi'(s)(\Psi(t_2)-\Psi(s))^{\mu-1}  (\Psi(s)-\Psi(0))^{1-\xi} g_1(s)\,ds\right.\\
&\left.\qquad-\frac{\left( \Psi \left( t_1\right) -\Psi \left( 0\right) \right)
^{1-\xi } }{\Gamma \left( \mu\right) }\int_{0}^{t_1}\Psi'(s)(\Psi(t_1)-\Psi(s))^{\mu-1} (\Psi(s)-\Psi(0))^{1-\xi} g_1(s)\,ds\right\vert\\
&\leq 
\left\vert   \frac{\left( \Psi \left( t_2\right) -\Psi \left( 0\right) \right)
^{1-\xi }\left\| g_1\right\|_{C_{1-\xi ;\, \Psi }\left( J,\mathbb{R}\right) }}{\Gamma \left( \mu\right) }\int_{0}^{t_2}\Psi'(s)(\Psi(t_2)-\Psi(s))^{\mu-1} \,ds\right.\\
&\left.\qquad-\frac{\left( \Psi \left( t_1\right) -\Psi \left( 0\right) \right)
^{1-\xi }\left\| g_1\right\|_{C_{1-\xi ;\, \Psi }\left( J,\mathbb{R}\right) } }{\Gamma \left( \mu\right) }\int_{0}^{t_1}\Psi'(s)(\Psi(t_1)-\Psi(s))^{\mu-1}\,ds\right\vert\\
&= 
  \frac{\left\| g_1\right\|_{C_{1-\xi ;\, \Psi }\left( J,\mathbb{R}\right) }}{\Gamma \left( \mu+1\right) }\left\lbrace  (\Psi(t_2)-\Psi(0))^{\mu+1-\xi} - (\Psi(t_1)-\Psi(0))^{\mu+1-\xi} \right\rbrace.
 \end{align*}

By the  continuity of  $\Psi$, from the above inequality it follows that
$$
\text{if}\left|t_1-t_2 \right|\to 0 ~\text{then}\left| \left( \Psi \left( t_2\right) -\Psi \left( 0\right) \right)
^{1-\xi }B_1\left(y, x\right)(t_2)-\left( \Psi \left( t_1\right) -\Psi \left( 0\right) \right) ^{1-\xi }B_1\left(y, x\right)(t_1)\right|\to 0,
$$ 
uniformly for all $\left(y, x\right)\in S^*$. Following the similar type of steps, we have 
$$
\text{if}\left|t_1-t_2 \right|\to 0 ~\text{then}\left| \left( \Psi \left( t_2\right) -\Psi \left( 0\right) \right)
^{1-\xi }B_2\left(y, x\right)(t_2)-\left( \Psi \left( t_1\right) -\Psi \left( 0\right) \right) ^{1-\xi }B_2\left(y, x\right)(t_1)\right|\to 0,
$$ 
uniformly for all $\left(y, x\right)\in S^*$.

From the parts (b) and (c), it follows that $B(S^*)$ is uniformly bounded and  equicontinous set in $E$. Then by Arzel$\acute{a}$-Ascoli theorem, $B(S^*)$ is relatively compact. Therefore, $B:S^*\rightarrow E$ is a compact operator. Since $B:S^*\rightarrow E$ is  continuous and compact operator, it is completely continuous.
\\ \\
\noindent\textbf{Step 3:} For $\left(y, x\right)\in E,$~$\left(y, x\right)=\left( A_1\left(y, x\right)\,B_1\left(\bar{y}, \bar{x}\right)+C_1\left(y, x\right), A_2\left(y, x\right)\,B_2\left(\bar{y}, \bar{x}\right)+C_2\left(y, x\right)\right) $
$\implies \left(y, x\right)\in S^*,~ \text{for all} ~\left(\bar{y}, \bar{x}\right)\in S^*$.

Let any $\left(y, x\right)\in E$ and $\left(\bar{y}, \bar{x}\right)\in S^*$   such that 
$$
\left(y, x\right)=\left( A_1\left(y, x\right)\,B_1\left(\bar{y}, \bar{x}\right)+C_1\left(y, x\right), A_2\left(y, x\right)\,B_2\left(\bar{y}, \bar{x}\right)+C_2\left(y, x\right)\right).
$$ 

Using the hypothesis {\bf (H4)} and boundedness of $u_1$ and $w_1$, for any $t\in J$, we have
\begin{align*}
&\left\vert \left( \Psi \left( t\right) -\Psi \left( 0\right) \right)
^{1-\xi }y(t) \right\vert\\
&=\left\vert \left( \Psi \left( t\right) -\Psi \left( 0\right) \right)
^{1-\xi }\left[ A_1\left(y, x\right)(t)\,B_1\left(\bar{y}, \bar{x}\right)(t) + C_1\left(y, x\right)(t)\right] \right\vert\\
&=\left\vert \left( \Psi \left( t\right) -\Psi \left( 0\right) \right)
^{1-\xi }\left[ u_1(t,y(t), x(t))\left\lbrace  \left( \Psi \left( t\right) -\Psi \left( 0\right) \right)^{\xi-1 }\,\bar{ \Omega}_1
+\mathcal{I}_{0^+}^{\mu\,;\, \Psi}v_1\left(t, \bar{y}(t), \bar{x}(t)\right)\right\rbrace\right.\right.\\
&\left.\left.~~+ w_1(t,y(t), x(t))\right]  \right\vert\\
&\leq\left\vert u_1(t,y(t), x(t))\right\vert \left\lbrace \left\vert  \bar{ \Omega}_1\right\vert+\frac{ \left( \Psi \left( t\right) -\Psi \left(0\right) \right)
^{1-\xi }}{\Gamma \left( \mu
\right) }\int_{0}^{t}\Psi'(s)(\Psi(t)-\Psi(s))^{\mu-1} \left\vert v_1\left(s, \bar{y}(s), \bar{x}(s)\right)\right\vert\,ds\right\rbrace\\
&\quad+ \left( \Psi \left( t\right) -\Psi \left( 0\right) \right)
^{1-\xi } \left\vert w_1(t,y(t), x(t))\right\vert \\
&\leq M_1 \left\lbrace \left\vert \bar{ \Omega}_1\right\vert+\frac{ \left( \Psi \left( t\right) -\Psi \left( 0\right) \right)
^{1-\xi }}{\Gamma \left( \mu
\right) }\int_{0}^{t}\Psi'(s)(\Psi(t)-\Psi(s))^{\mu-1} \left( \Psi \left( s\right) -\Psi \left( 0\right) \right) ^{1-\xi } g_1(s)\,ds\right\rbrace\\
&\quad+ \left( \Psi \left( t\right) -\Psi \left( 0\right) \right)
^{1-\xi } N_1 \\
&\leq M_1 \left\lbrace \left\vert  \bar{ \Omega}_1\right\vert+\frac{ \left( \Psi \left( T\right) -\Psi \left( 0\right) \right)^{\mu+1-\xi }\,}{\Gamma \left( \mu+1\right) }\left\| g_1\right\|_{C_{1-\xi ;\, \Psi }\left( J,\mathbb{R}\right) }\right\rbrace+ \left( \Psi \left( T\right) -\Psi \left( 0\right) \right)^{1-\xi } N_1,
\end{align*}
where
$$
\bar{ \Omega}_1=\frac{y_0- b\, \left( \Psi \left( T\right) -\Psi \left( 0\right) \right)^{1-\xi } \left(   w_1(T, \bar{y}(T), \bar{x}(T))-\,  u_1(T, \bar{y}(T), \bar{x}(T)) \,\mathcal{I}^{\mu;\, \Psi}_{0^+} v_1(T, \bar{y}(T), \bar{x}(T))\right)}{a\,u_1(0, \bar{y}(0+), \bar{x}(0+))+b\, u_1(T, \bar{y}(T), \bar{x}(T))}.
$$

This gives
\begin{equation}\label{eqq677}
\left\Vert y\right\Vert _{C_{1-\xi ;\,\Psi }\left( J,\mathbb{R}\right)  }\leq M_1 \left\lbrace \left\vert  \bar{ \Omega}_1\right\vert+\frac{ \left( \Psi \left( T\right) -\Psi \left( 0\right) \right)^{\mu+1-\xi }\,}{\Gamma \left( \mu+1\right) }\left\| g_1\right\|_{C_{1-\xi ;\, \Psi }\left( J,\mathbb{R}\right) }\right\rbrace+ \left( \Psi \left( T\right) -\Psi \left( 0\right) \right)^{1-\xi } N_1.
\end{equation}

Similarly, we can obtain
\begin{equation}\label{eqq678}
\left\Vert x\right\Vert _{C_{1-\xi ;\,\Psi }\left( J,\mathbb{R}\right)  }\leq M_2 \left\lbrace \left\vert  \bar{ \Omega}_2\right\vert+\frac{ \left( \Psi \left( T\right) -\Psi \left( 0\right) \right)^{\mu+1-\xi }\,}{\Gamma \left( \mu+1\right) }\left\| g_2\right\|_{C_{1-\xi ;\, \Psi }\left( J,\mathbb{R}\right) }\right\rbrace+ \left( \Psi \left( T\right) -\Psi \left( 0\right) \right)^{1-\xi } N_2, 
\end{equation}
where
$$
\bar{ \Omega}_2=\frac{y_0- b\, \left( \Psi \left( T\right) -\Psi \left( 0\right) \right)^{1-\xi } \left(   w_2(T, \bar{y}(T), \bar{x}(T))-\,  u_2(T, \bar{y}(T), \bar{x}(T)) \,\mathcal{I}^{\mu;\, \Psi}_{0^+} v_2(T, \bar{y}(T), \bar{x}(T))\right)}{a\,u_2(0, \bar{y}(0+), \bar{x}(0+))+b\, u_2(T, \bar{y}(T), \bar{x}(T))}.
$$

Using definition of norm on $E$ and the inequalities \eqref{eqq677} and \eqref{eqq678}, we obtain 
\begin{align*}
\left\Vert \left( y, x\right) \right\Vert _E&=\left\Vert y\right\Vert _{C_{1-\xi ;\,\Psi }\left( J,\mathbb{R}\right)  } + \left\Vert x \right\Vert _{C_{1-\xi ;\,\Psi }\left( J,\mathbb{R}\right)}\\
&\leq M_1 \left\lbrace \left\vert  \bar{ \Omega}_1\right\vert+\frac{ \left( \Psi \left( T\right) -\Psi \left( 0\right) \right)^{\mu+1-\xi }\,}{\Gamma \left( \mu+1\right) }\left\| g_1\right\|_{C_{1-\xi ;\, \Psi }\left( J,\mathbb{R}\right) }\right\rbrace+ \left( \Psi \left( T\right) -\Psi \left( 0\right) \right)^{1-\xi } N_1\\
&+ M_2 \left\lbrace \left\vert  \bar{ \Omega}_2\right\vert+\frac{ \left( \Psi \left( T\right) -\Psi \left( 0\right) \right)^{\mu+1-\xi }\,}{\Gamma \left( \mu+1\right) }\left\| g_2\right\|_{C_{1-\xi ;\, \Psi }\left( J,\mathbb{R}\right) }\right\rbrace+ \left( \Psi \left( T\right) -\Psi \left( 0\right) \right)^{1-\xi } N_2\\
&= M_1  \left\vert  \bar{ \Omega}_1\right\vert + M_2  \left\vert  \bar{ \Omega}_2\right\vert + \left( \Psi \left( T\right) -\Psi \left( 0\right) \right)^{1-\xi } \left[ N_1+N_2\right] \\
&+\frac{ \left( \Psi \left( T\right) -\Psi \left( 0\right) \right)^{\mu+1-\xi }\,}{\Gamma \left( \mu+1\right) }\left[ M_1\,\left\| g_1\right\|_{C_{1-\xi ;\, \Psi }\left( J,\mathbb{R}\right) } + M_2\,\left\| g_2\right\|_{C_{1-\xi ;\, \Psi }\left( J,\mathbb{R}\right) }\right] \\
&=R^*
\end{align*}

This implies, $ \left( y, x\right)\in S^*$.
\\ \\
\noindent\textbf{Step 4:} To prove $K M^* + L <1$, where $M^*=\sup\left\lbrace \left\|B\left( y, x\right)  \right\| _E: \left( y, x\right)\in  S^* \right\rbrace $.

Here,
\begin{align*}
M^*&=\sup\left\lbrace \left\|B\left( y, x\right) \right\| _E: \left( y, x\right)\in  S^* \right\rbrace\\
&=\sup\left\lbrace \left\|\left( B_1\left( y, x\right), B_2\left( y, x\right)\right)  \right\| _E: \left( y, x\right)\in  S^* \right\rbrace\\
&=\sup\left\lbrace ~\left\| B_1\left( y, x\right)\right\|_{C_{1-\xi ;\,\Psi }\left( J, \mathbb{R}\right)} + \left\|  B_2\left( y, x\right) \right\|_{C_{1-\xi ;\,\Psi }\left( J, \mathbb{R}\right)}: \left( y, x\right) \in  S^*\,~ \right\rbrace\\
&\leq \left\vert \Omega_1\right|+\left\vert \Omega_2\right| +\frac{\left( \Psi \left( T\right) -\Psi \left( 0\right) \right)^{\mu+1-\xi }}{\Gamma \left( \mu+1   \right) } \left(  \left\| g_1\right\|_{C_{1-\xi ;\, \Psi }\left( J, \mathbb{R}\right) } + \left\| g_2\right\|_{C_{1-\xi ;\, \Psi }\left( J, \mathbb{R}\right) }\right).
\end{align*}
Using the condition \eqref{eqq672}, we observe that 
\begin{small}
\begin{align*}
&K\, M^* + L\\
&\leq \left(\sum_{i=1}^{2} \sigma_i\right) \left[ \sum_{i=1}^{2}\left\vert \Omega_i\right|+\frac{\left( \Psi \left( T\right) -\Psi \left( 0\right) \right)^{\mu+1-\xi }}{\Gamma \left( \mu+1   \right) } \left(  \sum_{i=1}^{2}\left\| g_i\right\|_{C_{1-\xi ;\, \Psi }\left( J,\mathbb{R}\right) }\right)  \right]  +\sum_{i=1}^{2} \delta_i<1.
\end{align*}
\end{small}
From \textbf{steps 1 to 4}, it follows that all the conditions of Lemma \ref{hyb3} are fulfilled. Consequently,  by applying Lemma \ref{hyb3}, the operator equation $(y, x)=A(y, x)\,B(y, x)
+ C(y, x)$ has a  solution in  $ S^*$. Hence, the {\rm BVPs} for coupled system of hybrid {\rm FDEs} \eqref{eq663}-\eqref{eq664}  has a solution in $ C_{1-\xi ;\, \Psi }\left( J,\mathbb{R}\right)\times C_{1-\xi ;\, \Psi }\left( J,\mathbb{R}\right) $.
\end{proof}

\begin{rem}
\begin{enumerate}[topsep=0pt,itemsep=-1ex,partopsep=1ex,parsep=1ex]
\item If $a=1$ and $ b=0$, then the  {\rm BVP} for coupled system of  $\Psi$-Hilfer hybrid {\rm FDEs}  \eqref{eq663}-\eqref{eq664}  reduces to the  {\rm IVP} for coupled system of  $\Psi$-Hilfer hybrid {\rm FDEs}.
\item If $a=1, b=-1$ and $y_0=0$, then the  {\rm BVP} for coupled system of  $\Psi$-Hilfer hybrid {\rm FDEs}  \eqref{eq663}-\eqref{eq664}  reduces to the  periodic {\rm BVP} for coupled system of  $\Psi$-Hilfer hybrid {\rm FDEs}.
\item If $a=1, b=1$ and $y_0=0$, then the  {\rm BVP} for coupled system of  $\Psi$-Hilfer hybrid {\rm FDEs}  \eqref{eq663}-\eqref{eq664}  reduces to the  anti-periodic {\rm BVP} for coupled system of  $\Psi$-Hilfer hybrid {\rm FDEs}.

\end{enumerate}
\end{rem}
\section{Example}
In this section, to illustrate the obtained results, we provide two examples. To obtain exact numerical values, we   take  $\Psi(t)=t$ and $\nu=1$. Example \ref{ex1} illustrate the Theorem \ref{tha63.2} and   Example \ref{ex2} illustrate the  Theorem \ref{tha64.2}
\begin{ex}\label{ex1}
Consider the {\rm IVP} for coupled hybrid {\rm FDEs}  involving Caputo fractional derivative 
\begin{align}\label{eqq651}
\begin{cases}
& ^C \mathcal{D}^{\frac{1}{2}}_{0^+}\left[ \frac{\frac{7}{97}\left( y(t)-t\left[  y(t)+1-\frac{2}{t}\right] \right) }{\frac{1}{10} \left( t y(t)-2\right)} \right] 
   =\frac{x^2(t)}{1+x^2(t)}-\frac{\frac{3\sqrt{\pi}}{4} \,t^\frac{1}{2} \, \mathcal{I}^{\frac{1}{2}}_{0^+}x(t)}{\frac{3\sqrt{\pi}}{4}\,t^\frac{1}{2} \, \mathcal{I}^{\frac{1}{2}}_{0^+}x(t)+1},~a.e. ~t \in  (0,\,1],  ~\\
 &  y(0)=0,
\end{cases}
\end{align}
\begin{align}\label{eqq652}
\begin{cases}
   & ^C \mathcal{D}^{\frac{1}{2}}_{0^+}\left[ \frac{\frac{7}{97}\left(  x(t)-t\left[  x(t)+1-\frac{2}{t}\right] \right) }{\frac{1}{10} \left( t x(t)-2\right)} \right] 
      =\frac{y^2(t)}{1+y^2(t)}-\frac{\frac{3\sqrt{\pi}}{4} \,t^\frac{1}{2} \, \mathcal{I}^{\frac{1}{2}}_{0^+}y(t)}{\frac{3\sqrt{\pi}}{4} \,t^\frac{1}{2} \, \mathcal{I}^{\frac{1}{2}}_{0^+}y(t)+1},~a.e. ~t \in  (0,\,1],  ~\\
   &  \,x(0)=0.
\end{cases}
\end{align}
\end{ex}
Comparing the problem \eqref{eqq651} and \eqref{eqq652} with the coupled system of hybrid FDEs  \eqref{eqq611}-\eqref{eqq612}, we obtain
$$
\mu=\frac{1}{2},\, \nu=1,\, \xi=1,\,\Psi(t)=t,\, ~ y_0=0,~J=[0, 1],
$$
\begin{align*}
 u(t, y(t))&=\frac{1}{10}\left( t y(t)-2\right),\\
 v\left( t, y(t), k\,\mathcal{I}^{\mu}_{0^+}y(t)\right) &=\frac{y^2(t)}{1+y^2(t)}-\frac{\frac{3\sqrt{\pi}}{4}t^\frac{1}{2}  \mathcal{I}^{\frac{1}{2}}_{0^+}y(t)}{\frac{3\sqrt{\pi}}{4}t^\frac{1}{2}  \mathcal{I}^{\frac{1}{2}}_{0^+}y(t)+1},
\end{align*}
and
\begin{align*}
w(t, y(t))= \frac{7}{97}\,t\left[  y(t)+1-\frac{2}{t}\right].
\end{align*}
Note that for $\nu=1$, we have $ \xi=1$ and in  this case the space  $C_{1-\xi ;\, \Psi } \left( J, \R \right)$ is reduces to the space of continuous functions $C \left( J, \R \right)$.

Next, we prove that $u, v$ and $w$ satisfies the hypotheses (H1) and (H2) of the Theorem \ref{tha63.2}. For any $p, q \in \mathbb{R}$ and $t \in J$, we have
$$
|u(t, p)  - u(t, q)|=\left| \frac{1}{10}\left(  t\,p-2\right)  - \frac{1}{10}\left(   t\,q-2\right) \right| =  \frac{1}{10}\, t \left|\, p- \, q \right|\leq  \frac{1}{10} \left|\, p- \, q \right|,
$$
and  
$$
|w(t, p)-w(t, q)|=\left| \frac{7}{97}  t\left[  p+1-\frac{2}{t}\right]- \frac{7}{97}\, t\left[  q+1-\frac{2}{t}\right]\right| =  \frac{7}{97}\, t\left|\, p- \, q \right| \leq  \frac{7}{97} \left|\, p- \, q \right|.
$$
Thus the Lipschitz constants for the functions  $u$ and $w$ respectively are $\sigma= \frac{1}{10}=0.1$ and $\delta= \frac{7}{97}=0.07216$. Further, for any $p, q \in \mathbb{R}$ and $t \in J$, we have
\begin{align}
\left| v(t, p, q)\right| =\left| \frac{p^2}{1+p^2}-\frac{\frac{3\sqrt{\pi}}{4}t^\frac{1}{2}  q}{\frac{3\sqrt{\pi}}{4}t^\frac{1}{2} q+1}\right| \leq \left| \frac{p^2}{1+p^2}\right| +\left| \frac{\frac{3\sqrt{\pi}}{4}t^\frac{1}{2}  q}{\frac{3\sqrt{\pi}}{4}t^\frac{1}{2} q+1}\right|\leq 2=:g(t).
\end{align}
Therefore, 
\begin{align*}
4\, \sigma \left\lbrace \left| \frac{y_0\,}{u(0,y(0+))}\right| + \frac{1 }{\Gamma(\mu+1)}\left\| g\right\|_{C\left( J,\,\R\right) } \right\rbrace +\delta&=4\, \frac{1}{10} \left\lbrace \left| \frac{0\,}{u(0,y(0+))}\right| + \frac{1 }{\Gamma(\frac{1}{2}+1)}2 \right\rbrace +\frac{7}{97}\\
&\approx 0.9748<1.
\end{align*}
This implies the condition \eqref{616} is verified.
Since all the conditions of Theorem \ref{tha63.2} are satisfied, the coupled system of hybrid FDEs involving Caputo derivative \eqref{eqq651} and \eqref{eqq652} has at least one coupled solution in the space $C\left( J,\,\R\right) \times C\left( J,\,\R\right) $.
One can verify that $$(x,y)(t)= (t,t), ~t\in J $$ is a solution of the coupled system of hybrid FDEs \eqref{eqq651} and \eqref{eqq652}. 


Next, we  provide an example to illustrate the  Theorem \ref{tha64.2}.
\begin{ex}  \label{ex2}
Consider the {\rm BVP} for coupled hybrid {\rm FDEs}  involving Caputo fractional derivative 
\begin{align} \label{p54}
\begin{cases}
& ^C \mathcal{D}^{\frac{1}{3}}_{0^+}\left[ \frac{\frac{3}{17}\left( y(t)-\frac{17}{21}\left[  t\,y(t)+\frac{21}{17}x(t)+1\right] \right) }{\frac{1}{99} \left( \frac{ty(t)}{3}+\frac{tx(t)}{2}+\frac{5}{6}\right)} \right] 
   =\frac{e^{-t^2}}{97}\left[ \frac{y(t)}{2+y(t)}-\frac{x(t)}{2+x(t)}\right] ,\,~a.e. ~t \in  (0,\,1],  ~\\
 & 3\, y(0)+y(1)=1,
\end{cases}
\end{align}

\begin{align}\label{p55}
\begin{cases}
& ^C \mathcal{D}^{\frac{1}{3}}_{0^+}\left[ \frac{ x(t)-\left[ \frac{t}{10} \left( y(t)+x(t)+10\right) +2\right]  }{\frac{1}{98}\left[ \frac{ty(t)}{5}+ tx(t)+12\right] } \right] 
      =\frac{2^{-t}}{87}\left[ \frac{t^2-y(t)x(t)}{1-y(t)x(t)}\right],  \,~a.e. ~t \in  (0,\,1],  ~\\
 & 3\, x(0)+x(1)=1.
\end{cases}
\end{align}
\end{ex}
Comparing the  problem  \eqref{p54}-\eqref{p55} with the BVP for $\Psi$-Hilfer hybrid FDEs  \eqref{eq663}-\eqref{eq664}, we obtain
$$
\mu=\frac{1}{3},\, \nu=1,\, \xi=1,\,\Psi(t)=t,\, a=3, b=1, ~ y_0=1,~J=[0, 1],
$$
\begin{align*}
 u_1(t, y(t), x(t))&=\frac{1}{99} \left( \frac{ty(t)}{3}+\frac{tx(t)}{2}+\frac{5}{6}\right),~
u_2(t, y(t), x(t))=\frac{1}{98}\left[ \frac{ty(t)}{5}+ tx(t)+12\right],\\
v_1(t, y(t), x(t))&=\frac{e^{-t^2}}{97}\left[ \frac{y(t)}{2+y(t)}-\frac{x(t)}{2+x(t)}\right],~
v_2(t, y(t), x(t))=\frac{2^{-t}}{87}\left[ \frac{t^2-x(t)y(t)}{1-x(t)y(t)}\right],\\
w_1(t, y(t), x(t))&=\frac{1}{7}\left[  t\,y(t)+\frac{21}{17}x(t)+1\right] ~\mbox{and}\quad w_2(t, y(t), x(t))= \frac{t}{10}\left( y(t)+x(t)+10\right) +2.
\end{align*}
We prove that $u_i, v_i$ and $w_i(i=1,2)$ satisfies the hypotheses of the Theorem \ref{tha64.2}. Note that
\begin{align*}
 u_1(0, y(0),x(0))&=\frac{5}{594},~~u_2(0, y(0), x(0))=\frac{6}{49},\\
 ~u_1(1, y(1), x(1))&=\frac{5}{297},~~u_2(1, y(1), x(1))=\frac{33}{245},\\
v_1(1, y(1), x(1))&=0,~~
v_2(1, y(1), x(1))=0,\\
w_1(1, y(1), x(1))&=\frac{55}{119},~~\mbox{and}~~ w_2(1, y(1), x(1))= \frac{32}{10}.
\end{align*}
For any $p,\bar{p}, q, \bar{q} \in \mathbb{R}$ and $t \in J$, we have
\begin{align*}
|u_1(t, p, q)  - u_1(t, \bar{p},  \bar{q})|&=\left|\frac{1}{99} \left( \frac{t p}{3}+\frac{t q}{2}+\frac{5}{6}\right)  - \frac{1}{99} \left( \frac{t \bar{p}}{3}+\frac{t \bar{q}}{2}+\frac{5}{6}\right) \right|\\ &\leq  \frac{1}{99} \left\lbrace \left|\, p- \,  \bar{p} \right|+\left|\, q- \,  \bar{q} \right|\right\rbrace;
\end{align*} 
\begin{align*}
|u_2(t, p, q)  - u_2(t, \bar{p},  \bar{q})|&=\left|\frac{1}{98}\left[ \frac{t\,p}{5}+ t\,q+12\right] - \frac{1}{98}\left[ \frac{t\,\bar{p}}{5}+ t\,\bar{q}+12\right] \right| \\
&\leq \frac{1}{98} \left\lbrace \left|\, p- \,  \bar{p} \right|+\left|\, q- \,  \bar{q} \right|\right\rbrace;
\end{align*}
\begin{align*}
|w_1(t, p, q)  - w_1(t, \bar{p},  \bar{q})|&=\left|\frac{1}{7}\left[  t\,p+\frac{21}{17}q+1 \right]  - \frac{1}{7}\left[  t\,\bar{p} + \frac{21}{17}\bar{q}+1 \right] \right|\\ &\leq  \frac{2}{7} \left\lbrace \left|\, p- \,  \bar{p} \right|+\left|\, q- \,  \bar{q} \right|\right\rbrace;
\end{align*}
and
\begin{align*}
|w_2(t, p, q)  - w_2(t, \bar{p},  \bar{q})|&=\left|\left[ \frac{t}{10} \left( p+q+10\right) +2\right] - \left[ \frac{t}{10} \left( \bar{p}+\bar{q}+10\right) +2\right]\right|\\
& \leq \frac{1}{10} \left\lbrace \left|\, p- \,  \bar{p} \right|+\left|\, q- \,  \bar{q} \right|\right\rbrace.
\end{align*}
We have proved that the functions $u_i$ and $w_i ~(i=1,2)$ satisfies the Lipschitz type conditions. The Lipschitz  constants respectively are
$$\sigma_1= \frac{1}{99},~\sigma_2= \frac{1}{98},~\delta_1= \frac{2}{7},~\delta_2= \frac{1}{10}.$$  Next, for any $p, q \in \mathbb{R}$ and $t \in J$, we have
\begin{align}
\left| v_1(t, p, q)\right| =\left| \frac{e^{-t^2}}{97}\left[ \frac{p}{2+p}-\frac{q}{2+q}\right]\right|\leq \frac{2}{97}:= g_1(t)
\end{align}
and 
\begin{align}
\left| v_2(t, p, q)\right| =\left|\frac{2^{-t}}{87}\left[ \frac{t^2-p\,q}{1-x(t)y(t)p\,q}\right]\right|\leq \frac{1}{87}:=g_2(t). 
\end{align}
 Using the above calculated values in Eqs.\eqref{ie61} and \eqref{ie62}, we obtain 
 $$\Omega_1=\frac{38016}{2975},~ \Omega_2=\frac{-539}{123}.$$
Therefore, 
\begin{align*}
&\left(\sum_{i=1}^{2} \sigma_i\right) \left[ \sum_{i=1}^{2}\left\vert \Omega_i\right|+\frac{\left( \Psi \left( T\right) -\Psi \left( 0\right) \right)^{\mu+1-\xi }}{\Gamma \left( \mu+1   \right) } \left(  \sum_{i=1}^{2}\left\| g_i\right\|_{C_{1-\xi ;\, \Psi }\left( J,\mathbb{R}\right) }\right)  \right]  +\sum_{i=1}^{2} \delta_i\\
&=\left( \frac{1}{99}+\frac{1}{98}\right)\left[ \frac{38016}{2975}+\frac{539}{123}+\frac{1}{\Gamma\left(\frac{4}{3} \right) } \left(\frac{2}{97}+\frac{1}{87} \right) \right]+\left( \frac{2}{7}+\frac{1}{10}\right)  \\
&\approx 0.7348<1.
\end{align*}
Hence, the condition \eqref{eqq672} is verified. 

Since all the conditions of Theorem \ref{tha64.2} are satisfied,  the {\rm BVPs} for coupled system of hybrid {\rm FDEs} \eqref{p54} and \eqref{p55}  has at least one solution in the space $C\left( J,\,\R\right) \times C\left( J,\,\R\right) $. One can verify that $$(x,y)(t)= (t,t), ~t\in J, $$ is a solution of the {\rm BVPs} for coupled system of hybrid {\rm FDEs} \eqref{p54} and \eqref{p55}.


\section*{Conclusion}
The existence  of solution  of IVP and  {\rm BVP} for coupled system of $\Psi$-Hilfer hybrid {\rm FDEs} is achieved by using fixed point theorem for the three operators. It is observed that the existence result obtained for the BVPs of coupled system of   $\Psi$-Hilfer hybrid {\rm FDEs}  includes the study of coupled system for IVP $(a=1, b=0)$, periodic BVP $(a=1, b=-1, y_0=0)$ and anti-periodic BVP $(a=1, b=1, y_0=0)$ involving $\Psi$-Hilfer  fractional derivative. Further, we have provided an examples to illustrate the validity of our outcomes.

As presented in the body of the paper, we have successfully obtained the main results of this paper. However, some open problems that still need to be answered that involve the theory of fractional hybrid differential equations, namely:
\begin{enumerate}[topsep=0pt,itemsep=-1ex,partopsep=1ex,parsep=1ex]
\item Would it be possible to discuss the existence of mild solutions to Eqs.(\ref{eqq611})-(\ref{eqq612}) problems? What are the necessary and sufficient conditions for this to happen?

\item As a consequence of item 1, we can ask about the uniqueness and stability of mild solutions.
    
\item Is it possible to guarantee solutions involving sectorial and almost-sectorial operators?
    
\end{enumerate}

There are some questions that need to be answered as outlined above, which will enrich the theory. Other questions about fractional hybrid differential equations, are being discussed and future works are being elaborated, which allowed to answer these questions and others that are still open.

\section*{Acknowledgment}
The second author  acknowledges the Science and Engineering Research Board (SERB), New Delhi, India for the Research Grant (Ref: File no. EEQ/2018/000407).

\section*{Declaration of interests}
The authors declare that they have no known competing financial interests or personal relationships that could have appeared to influence the work reported in this paper.

\section*{Credit author statement}
All authors contributed equally to this work.


\end{document}